\documentclass[10.5pt]{article}
\usepackage{mathrsfs}
\usepackage{amsfonts}
\usepackage{enumerate}
\usepackage{cite} % Make references as [1-4], not [1,2,3,4]
\usepackage{url}  % Formatting web addresses
\usepackage{ifthen}  % Conditional
\usepackage{multicol}   %Columns
\usepackage[utf8]{inputenc} %unicode support
\usepackage{amsthm}
\usepackage{amsmath}
\usepackage{amsfonts}
\usepackage{amssymb}
\usepackage{bm}
\usepackage{bbm}
\usepackage{latexsym}
\usepackage{indentfirst}
\usepackage{graphicx}
\usepackage{leftidx}
\usepackage{epsfig}
\usepackage{epstopdf}
\usepackage{float}
\usepackage{geometry}
\usepackage{rotating}
\usepackage{tikz}
\usetikzlibrary{decorations.pathmorphing} % for snake lines
\usetikzlibrary{matrix} % for block alignment
\usetikzlibrary{arrows} % for arrow heads
\usetikzlibrary{calc} % for manimulation of coordinates

\tikzstyle{block} = [draw,rectangle,thick,minimum height=2em,minimum width=2em]
\tikzstyle{sum} = [draw,circle,inner sep=0mm,minimum size=2mm]
\tikzstyle{connector} = [->,thick]
\tikzstyle{line} = [thick]
\tikzstyle{branch} = [circle,inner sep=0pt,minimum size=1mm,fill=black,draw=black]
\tikzstyle{guide} = []
\tikzstyle{snakeline} = [connector, decorate, decoration={pre length=0.2cm,
                         post length=0.2cm, snake, amplitude=.4mm,
                         segment length=2mm},thick, magenta, ->]

\makeatletter
\def\leftharpoonfill@{\arrowfill@\leftharpoonup\relbar\relbar}
\def\rightharpoonfill@{\arrowfill@\relbar\relbar\rightharpoonup}
\newcommand\rbjt{\mathpalette{\overarrow@\rightharpoonfill@}}
\newcommand\lbjt{\mathpalette{\overarrow@\leftharpoonfill@}}
\makeatother

\theoremstyle{plain}
\newtheorem{theorem}{Theorem}[section]
\newtheorem{lemma}{Lemma}[section]
\newtheorem{proposition}{Proposition}[section]

\newtheorem{remark}{Remark}[section]

\theoremstyle{definition}
\newtheorem{definition}{Definition}[section]

\def\includegraphics{}

\def\wt{\widetilde}

\def\tilde{\widetilde}
\def\hat{\widehat}
\def\bar{\overline}

\def\cd{\cdot}

\def\e{\varepsilon}
\def\z{\zeta}

\def\l{\lambda}

\def\cA{{\cal A}}
\def\cB{{\cal B}}
\def\cC{{\cal C}}
\def\cD{{\cal D}}

\def\cF{{\cal F}}
\def\cG{{\cal G}}

\def\cJ{{\cal J}}

\def\cQ{{\cal Q}}
\def\cR{{\cal R}}
\def\cS{{\cal S}}

\def\cU{{\cal U}}
\def\cV{{\cal V}}

\def\hE{\mathbb{E}}

\def\hH{\mathbb{H}}

\def\hP{\mathbb{P}}

\def\hR{\mathbb{R}}

\def\ud{\mathrm{d}}

\def\qq{\qquad}
\def\q{\quad}

\def\QEDopen{{\setlength{\fboxsep}{0pt}\setlength{\fboxrule}{0.2pt}\fbox{\rule[0pt]{0pt}{1.3ex}\rule[0pt]{1.3ex}{0pt}}}}
\def\QED{\QEDopen}

\def\endproof{\hspace*{\fill}~\QED\par\endtrivlist\unskip}

\newenvironment{keywords}{{\bf Key words: }}{}

\begin{document}

\title{Linear-Quadratic Mixed Stackelberg-Nash Stochastic Differential Game with Major-Minor Agents\thanks{J. Huang acknowledges the financial support by RGC Grant PolyU 153005/14P, 153275/16P.}}

\author{Kehan Si\thanks{School of Mathematics, Shandong University, Jinan, Shandong Province, 250100, China (sikehan@mail.sdu.edu.cn).} \, \  James Huang\thanks{Department of Applied Mathematics, The Hong Kong Polytechnic University, Hong Kong, China (james.huang@polyu.edu.hk).} \, \  Zhen Wu\thanks{School of Mathematics, Shandong University, Jinan, Shandong Province, 250100, China (wuzhen@sdu.edu.cn).}}

\maketitle

\begin{abstract}We consider a controlled linear-quadratic (LQ) large-population system with mixture of three types agents: major leader, minor leaders and minor followers. The Stackelberg-Nash-Cournot (SNC) approximate equilibrium is studied by a major-minor mean-field game (MFG) coupled with a leader-follower Stackelberg game. By variational method, the SNC approximate equilibrium strategy can be represented by some forward-backward-stochastic-differential-equations (FBSDEs) in the open-loop sense. And we pay great effort to give the feedback form of the open-loop strategy by some Riccati equations. \end{abstract}


\section{Introduction}On a given finite time horizon $[0,T],$ let $(\Omega,\cF, \mathbb{P})$ be a complete probability space on which a $(1+N_l+N_f)$-dimensional standard Brownian motion $\{W_0(t),W_i(t)$, $\widetilde{W}_j(t)\}_{0 \leq t \leq T}$ is defined. %%%%The filtration $\{\cF_t\}$ will be specified explicitly later when combining with other possible information structures.%%%
In this paper, we consider a large-population system involving $(1+N_l+N_f)$ individual agents (where $N_l$ and $N_f$ are very large) which are mixed with three types: the major-leader, denoted by $\cA_0$, minor-leaders $\cA_i^l, 1\leq i\leq N_l$ and the followers $\cA_j^f, 1\leq j\leq N_f$. The dynamics of $\cA_0,\{\cA_i^l\}_{i=1}^{N_{l}}, \{\cA_j^f\}_{j=1}^{N_{f}}$ are given respectively by the following controlled linear stochastic differential equations:
\begin{equation}\cA_0:\left\{\begin{aligned}\label{C6e1}
&\ud X_0(t)=\{A_0X_0(t)+B_0u_0(t)+E_0^1X^{(N_l)}(t)+F_0^1x^{(N_f)}(t)\}\ud t\\
&\qq\qq+\{C_0X_0(t)+D_0u_0(t)+E_0^2X^{(N_l)}(t)+F_0^2x^{(N_f)}(t)\}\ud W_0(t), \\
& X_0(0)=\xi_0,
\end{aligned}\right.\end{equation}
\begin{equation}\cA_i^l:\ i=1,2,\ldots,N_l,\left\{\begin{aligned}\label{C6e2}
&\ud X_i(t)=\{AX_i(t)+Bu_i(t)+E_1X^{(N_l)}(t)\}\ud t\\
&\qq\qq+\{CX_i(t)+Du_i(t)+E_2X^{(N_l)}(t)\}\ud W_i(t)\\
& X_i(0)=\xi_i
\end{aligned}\right.\end{equation}
and
\begin{equation}\cA_j^f:\ j=1,2,\ldots,N_f,\left\{\begin{aligned}\label{C6e3}
&\ud x_j(t)=\{\tilde{A}x_j(t)+\tilde{B}v_j(t)+F_1x^{(N_f)}(t)\}\ud t\\
&\qq\qq+\{\tilde{C}x_j(t)+\tilde{D}v_j(t)+F_2x^{(N_f)}(t)\}\ud\widetilde{W}_j(t)\\
& x_j(0)=\zeta_j,
\end{aligned}\right.\end{equation}
where $X^{(N_l)}(t)=\frac{1}{N_l}\sum_{i=1}^{N_l}X_i(t)$ and $x^{(N_f)}(t)=\frac{1}{N_f}\sum_{j=1}^{N_f}x_j(t)$ are called \emph{state-average} or \emph{mean field term}. $A_0$, $A$, $\tilde{A}$, $B_0$, $B$, $\tilde{B}$, $C_0$, $C$, $\tilde{C}$, $D_0$, $D$, $\tilde{D}$, $E_0^1$, $E_0^2$, $E_1$, $E_2$, $F_0^1$, $F_0^2$, $F_1$, $F_2$ are deterministic constant matrices with proper dimensions. In the above, $X_0(\cdot)$, $X_i(\cdot)$, $x_j(\cdot)$ are called the \emph{state process} taking values in $\hR^n$ with \emph{initial values} $\xi_0$, $\xi_i$, $\zeta_j$ which are random variables. $u_0(\cdot)$, $u_i(\cdot)$, $v_j(\cdot)$ are called \emph{admissible controls} taken by $(1+N_l+N_f)$ players in the game and taking values in $\hR^{m_1}$, $\hR^{m_2}$, $\hR^{m_3}$, respectively. Under some mild conditions on the coefficients, for any initial values $\xi_0$, $\xi_i$, $\zeta_j$, \eqref{C6e1}, \eqref{C6e2} and \eqref{C6e3} admits a unique strong solution. The performance can be measured by the following \emph{cost functionals}: for $\cA_0,$
\begin{equation}\begin{aligned}\label{C6e4}
\cJ_0(u_0(\cdot),\textbf{u}(\cdot),\textbf{v}(\cdot))=&\frac{1}{2}\mathbb{E}\Big\{\int_0^T \Big(\Big\|X_0(t)-\big(\lambda_0X^{(N_l)}(t)+(1-\lambda_0)x^{(N_f)}(t)\big)\Big\|_{Q_0}^2\\
&\qq\qq+\|u_0(t)\|_{R_0}^2\Big)\ud t+\|X_0(T)\|_{H_0}^2\Big\};
\end{aligned}\end{equation}for $\cA_i^l,\ 1\leq i \leq N_l$;\begin{equation}\begin{aligned}\label{C6e5}
\cJ_i^l(u_0(\cdot),u_i(\cdot),\textbf{u}_{-i}(\cdot))=&\frac{1}{2}\mathbb{E}\Big\{\int_0^T \Big(\Big\|X_i(t)-\big(\lambda X^{(N_l)}(t)+(1-\lambda)X_0(t)\big)\Big\|_{Q}^2\\
&\qq\qq+\|u_i(t)\|_{R}^2\Big)\ud t+\|X_i(T)\|_{H}^2\Big\},
\end{aligned}\end{equation}and for $\cA_j^f,\ 1\leq j \leq N_f$,
\begin{equation}\begin{aligned}\label{C6e6}
\cJ_j^f(u_0(\cdot),\textbf{u}(\cdot),v_j(\cdot),\textbf{v}_{-j}(\cdot))=&\frac{1}{2}\mathbb{E}\Big\{\int_0^T \Big(\Big\|x_j(t)-\big(\tilde{\lambda}_1 X_0(t)+\tilde{\lambda}_2 X^{(N_l)}(t)+\tilde{\lambda}_3 x^{(N_f)}(t)\big)\Big\|_{\tilde{Q}}^2\\
&\qq\qq+\|v_j(t)\|_{\tilde{R}}^2\Big)\ud t+\|x_j(T)\|_{\tilde{H}}^2\Big\}.
\end{aligned}\end{equation}where for given vector $z$, $\|z\|^2_{M}=\langle Mz, z\rangle$ for $M$ is any matrix or matrix-valued function of suitable dimensions, and where $Q_0,Q,\tilde{Q},R_0,R,\tilde{R},H_0,H,\tilde{H}$ are deterministic symmetric matrix of suitable dimensions. As we can see, all agents are coupled not only in their state process but also in their cost functionals with convex combinations of state-average.

Roughly speaking, agent $\cA_j^f$ will give his/her best respond according to the strategies from major leader $\cA_0$ and minor leaders $\cA_i^l$ to minimize his/her own cost functional $\cJ_j^f(u_0(\cdot),\textbf{u}(\cdot),v_j(\cdot),\textbf{v}_{-j}(\cdot))$. And agent $\cA_0$ will also give his/her best respond according to the strategies minor leaders $\cA_i^l$ and the best respond of minor followers to minimize his/her own cost functional $\cJ_0(u_0(\cdot),\textbf{u}(\cdot), \textbf{v}(\cdot))$. Knowing the best respond of major leader and minor followers, agent $\cA_i^l$ wants to minimize his/her own cost functional $\cJ_i^l(u_0(\cdot),u_i(\cdot),\textbf{u}_{-i}(\cdot))$ by choosing an optimal control $u_i(\cdot)$. However, due to the state-average coupling, our problem is essentially a high-dimensional Stackelberg-Nash differential game. Moreover, $\mathcal{A}_0$ is the the dominate or major leader because it effects the cost functionals of all minor leaders.

We call the above problem formulated as \emph{Mixed Stakelberg-Nash Major-minor (SN-MM) differential game}. The following comments on our formulation further verify such terminology.

(\emph{Single leader-follower game}) In case $N_{l}=0, N_{f}=1,$ thus there has no minor leaders and only single followers, with one major leader, then our problem reduces to the classical single-leader and single-follower game. The Stackelberg game has been proposed in 1934 by H. von Stackelberg \cite{S34}, when he defined a concept of a hierarchical solution for markets where some firms have power of dominating over others. This solution concept is now known as the Stackelberg equilibrium. Early study for stochastic Stackelberg differential games (SSDG) can be seen in Basar (1979) \cite{B79}. A pioneer work was done by Yong (2002) \cite{Y02-LQ}, where a LQ leader-follower stochastic differential game (SDG) was introduced and studied. The coefficients of the system and the cost functionals are random, the controls enter the diffusion term of the state equation, and the weight matrices for the controls in the cost functionals are not necessarily positive definite. To give a state feedback representation of the OL Stackelberg equilibrium, the related Riccati equations are derived and sufficient conditions for the existence of their solution with deterministic coefficients are discussed. Here after, Bensoussan, Chen, and Sethi (2015) \cite{BCS15} obtained the global maximum principles for both open-loop (OL) and closed-loop (CL) SSDG whereas the diffusion term does not contain the controls. The solvability of related Riccati equations is discussed, in order to obtain the state feedback Stackelberg equilibrium.

(\emph{Multiple leaders-followers game}) In case $N_{l}, N_f$ are of medium or small size, then our problem is reduced to the Stackelberg game with multiple leaders and multiple followers. It is a natural extension of the single leader-follower game and the relevant works include \cite{BCY15,BCY16,NCM12}, etc.

(\emph{Mean-field-game with symmetric agents}) In case $N_{l}=0,$ and no $\mathcal{A}_0$ involved, then our problem becomes the standard dynamic game with a very large number of minor (symmetric) agents in which each single agent interacts with the mass-effect of other agents only through coupling in states/dynamics. For large population stochastic dynamics, one effective method is to search its decentralized strategies by the mean-field-game (MFG) theory. We recall that there are much work to study mean field game (MFG). Since the recent independent works by Huang, Caines, and Malham\'{e} \cite{HCM03,HMC06} and Lasry and Lions \cite{LLJ06I,LLJ06II,LL07}, mean field game (MFG) theory and its applications have enjoyed rapid growth. MFG provides a simpler alternative framework for tackling the interactive game for a large number of homogeneous agents. By allowing agents to interact through a common medium, known as the mean field term, formulation of the dynamic game under the MFG framework consists of only a few equations. Further developments on the theory of MFG can be found in the works of Andersson and Djehiche \cite{AD11}, Bardi \cite{B13}, Bensoussan, Frehse, and Yam \cite{BFY13}, Buckdahn et al. \cite{BDLP07}, Cardaliaguet \cite{C12}, Carmona and Delarue \cite{CD12}, Garnier, Papanicolaou, and Yang \cite{GPY13}, Gu\'{e}ant, Lasry, and Lions \cite{GLL11}, Meyer-Brandis, \O ksendal, and Zhou \cite{MOZ}, and the references therein.

(\emph{Major-minor game}) In case $N_f=0,$ then there has no followers and only major and minor leaders involved, and our problem becomes the major-minor (MM) mean-field game (MFG). The MM-MFG is introduced in \cite{H10}, and has been well investigated by \cite{NC12} common major-minor mean field LQG game (refer to \cite{NC12}). Our model generalizes \cite{BCY15,BCY16,NC12,NCM12} because it includes not only leader-follower structure but also major-minor structure.

(\emph{Convex combination}) Refer to Nourian, Caines, Malham\'{e} and Huang (2012) \cite{NCM12}, here we consider a kind of general case of cost functional with likelihood ratio (i.e. convex combination). On other words, for an example, the cost functional of the major leaders is based on a trade-off between keeping cohesion of the flock of minor leaders and keeping cohesion of the flock of the followers (see \eqref{C6e4}). By the way, we may be interested in special case like $\tilde{\lambda}_3\neq0$ which means the cost functional of followers are directly influenced by the major-leader or $\tilde{\lambda}_3=0$ which means the cost functional of followers are not directly influenced by the major-leader. We will discuss difference between the special case and the general case in following sections.

\begin{remark}Application of our problem formulation may be found in power markets involving large size of consumers and large utilities together with the following producer; inventory management without stocking capacities. (refer to \cite{DX09})
The state processes are characterized by three kinds of group. One we called major leader agent can be regarded as the government or supervisory in the economic issues. And the ones we called minor leader agents can be regarded as the corresponding companies or firms. The rest ones we called minor follower agents can be regarded as the related suppliers of raw material or manufacturers of primary commodity, etc. We can see that the state processes of three types of group have no influence on each other but the cost functionals do have direct influence on each other.
\end{remark}

Our present work considers the combination problems of leader-follower and major-minor systems, where the large scale population is also under consideration. In the entire system, the major and a part of minor agents are together regarded as the leaders, which are called major-leader and minor-leaders respectively and the rest are called minor followers (followers). Obviously, the more complex structure will bring some technical problem. Besides, there are lots of interesting questions remain to be solved. For an example, we can consider the state processes include the mean-field term which may coincide the real world much better or focus on the more realistic cost functional, etc.

Let us now explain the argument structure of our problem. In principle, the above problem can be studied as a MM-MFG coupled with a leader-follower game. Accordingly, original problem can be analyzed through the following structures:

\textbf{Step 1}: Fix the mass effect limit of minor leaders $\bar{m}_{X}$ and major leader $(x_0, u_0).$ With frozen $(x_0, u_0, \bar{m}_{X}),$ introduce and solve the auxiliary problem to get the best decentralized response function of minor followers, denoted by $\bar{m}_{x}=\bar{m}_{x}(x_0, u_0, \bar{m}_{X}).$  %%%%%%For any choices $u_0(\cdot)$, $u(\cdot)$ of agents $\cA_0$, $\cA_i^l$, $1\leq i\leq N_l$, and fixed initial states $x_0,x,y$, agents $\cA_j^f$, $1\leq j\leq N_f$, would like to choose $\bar{v}(\cdot)$ such that for each $1\leq j\leq N_f$, $\cJ_j^f(x_0,x,y;u_0(\cdot),u(\cdot),\bar{v}_j(\cdot),v_{-j}(\cdot))$ is the minimum of $\cJ_j^f(x_0,x,y;u_0(\cdot),u(\cdot),v_j(\cdot), v_{-j}(\cdot))$.

\textbf{Step 2}: Given the response functional $\bar{m}_{x},$ and frozen $\bar{m}_{X},$ solve the decentralized SOC problem of $\mathcal{A}_0,$ and denote the optimal solution pair as $(\bar{X}_0,\bar{u}_0)$.  %%%%.Knowing the followers would take such an optimal control $\bar{v}(\cdot)$ (supposing it exists, which depends on the choices $u_0(\cdot)$ and $u(\cdot)$ of leaders and the initial state $x_0,x,y$, in general), agents $\cA_0$ (the major leader) would like to choose some $\bar{u}_0(\cdot)$ to minimize $\cJ_0(x_0,x,y;u_0(\cdot),u(\cdot),\bar{v}(\cdot))$.

\textbf{Step 3}: Given $\bar{m}_{x},$ solve the optimal control for the minor leaders. Influenced by the optimal control of major leader $\bar{u}_0(\cdot)$ (supposing it exists, which depends on the choices $\textbf{u}(\cdot)$ of minor leaders and the initial state $\xi_0$, $\xi_i$, $\zeta_j$, in general), agents $\cA_i^l$, $1\leq i\leq N_l$, (the minor leaders) would like to choose some $\bar{u}_i(\cdot)$ to minimize $\cJ_i^l(\bar{u}_0(\cdot),u_i(\cdot),\textbf{u}_{-i}(\cdot))$.

\textbf{Step 4} CC condition to specify $\bar{m}_{X}$ and all decentralized strategies can be designed. Approximate Stackelberg-Counot-Nash equilibrium can be verified.

The main contribution of this paper can be summarized as follows: \begin{itemize} \item The decentralized strategy profile is investigated in both (semi-)closed-loop and open-loop sense. \item Existence and uniqueness of the CC condition system is investigated in the global solvability case. \item the CC condition system is represented via a full-coupled mean-field type FBSDE in open-loop case, and FBSDE and non-standard Riccati equation in closed-loop sense. \item The approximate Nash equilibrium Stakleberg game is verified under more general condition (more than standard assumption with positive-definitiveness on coefficient matrix).  \end{itemize}

The rest of this paper is organized as follows. In section 2, we give the formal problem formulation and some preliminaries. In section 3, we discuss the open-loop strategy of Stackelberg mixed major-minor games. In section 4 We get the consistency condition system equations based on the open-loop strategy, which is a fully coupled FBSDE. Besides, we get the criteria to judge the well-posedness of such a FBSDE. At last, we verify the OL strategy we got is $\e$-Nash equilibrium OL strategy of the original problem.

\section{Preliminary and formulation}The following notations will be used throughout this paper. Let $\hR^{n}$ denotes the $n-$dimensional Euclidean space, $\hR^{n\times m}$ be the set of all $(n\times m)$ matrices, and let $\cS^n$ be the set of all $(n\times n)$ symmetric matrices. We denote the transpose by subscript $^\top$, the inner product by $\langle\cd,\cd\rangle$ and the norm by $|\cd|$. For $t\in[0,T]$ and Euclidean space $\hH$, we introduce the following function spaces:\begin{equation*}\begin{aligned}
&L^p(t,T;\hH)=\Big\{\psi:[t,T]\rightarrow\hH\ \Big|\ \int_{t}^{T}|\psi(s)|^p\ud s<\infty\Big\},\qq 1\leq p<\infty,\\
&L^\infty(t,T;\hH)=\Big\{\psi:[t,T]\rightarrow\hH\ \Big|\ \mbox{esssup}_{s\in[t,T]}|\psi(s)|<\infty\Big\},\\
&C([t,T];\hH)=\Big\{\psi:[t,T]\rightarrow\hH\ \Big|\ \psi(\cdot)\ \mathrm{is}\ \mathrm{continuous}\Big\}.\\
\end{aligned}\end{equation*}and the spaces of process or random variables on given filtrated probability space: \begin{equation*}\begin{aligned}
&L^2_{\cF_t}(\Omega;\hH)=\Big\{\xi:\Omega\rightarrow\hH\ \Big|\ \xi\ \mathrm{is}\ \cF_t\mathrm{-measurable},\ \hE[|\xi|^2]<\infty\Big\},\\
&L^2_{\cF}(t,T;\hH)=\Big\{\psi:[t,T]\times\Omega\rightarrow\hH\ \Big|\ \psi(\cdot)\ \mathrm{is}\ \cF_t\mathrm{-progressively}\ \mathrm{measurable},\ \hE\Big[\int_{t}^{T}|\psi(s)|^2\ud s\Big]<\infty\Big\},\\
&L^2_{\cF}(\Omega;C([t,T];\hH))=\Big\{\psi:[t,T]\times\Omega\rightarrow\hH\ \Big|\ \psi(\cdot)\ \mathrm{is}\ \cF_t\mathrm{-adapted},\ \mathrm{continuous},\ \hE\Big[\sup_{t\leq s\leq T}|\psi(s)|^2\Big]<\infty\Big\},\\
&L^2_{\cF}(\Omega;L^1(t,T;\hH))=\Big\{\psi:[t,T]\times\Omega\rightarrow\hH\ \Big|\ \psi(\cdot)\ \mathrm{is}\ \cF_t\mathrm{-progressively}\ \mathrm{measurable},\ \hE\Big[\int_{t}^{T}|\psi(s)|\ud s\Big]^2<\infty\Big\}.\\
\end{aligned}\end{equation*}We set the following information structures, which are important to introduce our admissible strategies: $\{\cF_t\}_{0\leq t\leq T}$ is the natural filtration generated by all BM components $\{W_0(\cdot),W_i(\cdot),\widetilde{W}_j(\cdot)\}$ augmented by all the $\mathbb{P}$-null sets in $\cF$, $1\leq i\leq N_l,\ 1\leq j\leq N_f$, it can be viewed as the full information of all states and noises; $\{\cF_t^0\}_{0 \leq t \leq T}$ is the natural filtration generated by $\{W_0(\cdot),X_0(\cdot)\}$ augmented by all the $\mathbb{P}$-null sets in $\cF$. It is the space on which the limiting state-average should be adapted; $\{\cF_t^i\}_{0 \leq t \leq T}$ is the natural filtration generated by $\{W_i(\cdot), X_i(\cdot)\}$ augmented by all the $\mathbb{P}$-null sets in $\cF$, $1\leq i\leq N_l$; $\{\cG_t^j\}_{0 \leq t \leq T}$ is the natural filtration generated by $\{\widetilde{W}_j(\cdot),Y_j(\cdot) \}$ augmented by all the $\mathbb{P}$-null sets in $\cF$, $1\leq j\leq N_f$.

Given information structures, we can set the following Hilbert spaces for \emph{centralized} and \emph{decentralized} strategies for individual agents in open-loop sense:
\begin{equation}\begin{aligned}\label{control}
\cU_0^c[0,T]\triangleq&L^2_{\cF}(0,T;\hR^{m_1}),\\
\cU_i^c[0,T]\triangleq&L^2_{\cF}(0,T;\hR^{m_2}),\q i=1,2,\ldots,N_l,\\
\cV_j^c[0,T]\triangleq&L^2_{\cF}(0,T;\hR^{m_3}),\q j=1,2,\ldots,N_f,\\
\end{aligned}\end{equation}and \emph{decentralized} open-loop strategies: \begin{equation}\begin{aligned}\label{control}
\cU_0^d[0,T]\triangleq&L^2_{\cF^0}(0,T;\hR^{m_1}),\\
\cU_i^d[0,T]\triangleq&L^2_{\cF^i}(0,T;\hR^{m_2}),\q i=1,2,\ldots,N_l,\\
\cV_j^d[0,T]\triangleq&L^2_{\cG^j}(0,T;\hR^{m_3}),\q j=1,2,\ldots,N_f.\\
\end{aligned}\end{equation}Let $(u_0,\textbf{u},\textbf{v})=(u_0,u_1, \ldots, u_{N_l},v_1,\ldots,v_{N_f})$ denote the strategy set of all $(1+N_l+N_f)$ agents; $\textbf{u}=(u_1, \ldots, u_{N_l})$ the set of control strategies of all $N_l$ major-leader agents; $\textbf{v}=(v_1, \ldots, v_{N_f})$ the set of strategy profile of all $N_f$ follower agents; $\textbf{u}_{-i}=(u_1, \ldots, u_{i-1},$ $u_{i+1}, \ldots,$ $u_{N_l})$ the control strategy set of major-leader agents except $\cA_i^l$; $\textbf{v}_{-j}=(v_1, \ldots, v_{j-1},$ $v_{j+1}, \ldots,$ $v_{N_f})$ the control strategy set of follower agents except the $j^{th}$ follower agent $\cA_j^f$.

Sometimes, when defining the Stakelberg-Nash-Counor strategy, it is helpful to set the following product space for strategy set. We denote
\begin{equation*}\left\{\begin{aligned}
&\cU^c[0,T]=\cU_1^c[0,T]\times\cdots\times\cU_{N_l}^c[0,T],&&\cV^c[0,T]=\cV_1^c [0,T]\times\cdots\times\cV_{N_f}^c[0,T],\\
&\cU^d[0,T]=\cU_1^d[0,T]\times\cdots\times\cU_{N_l}^d[0,T],&&\cV^d[0,T]=\cV_1^d [0,T]\times\cdots\times\cV_{N_f}^d[0,T],\\
\end{aligned}\right.\end{equation*}
\begin{equation*}\left\{\begin{aligned}
&\cU_{-i}^c[0,T]=\cU_1^c[0,T]\times\cdots\times\cU_{i-1}^c[0,T]\times\cU_{i+1}^c[0,T]\times\cdots\times\cU_{N_l}^c[0,T],\\
&\cV_{-j}^c[0,T]=\cV_1^c[0,T]\times\cdots\times\cV_{j-1}^c[0,T]\times\cV_{j+1}^c[0,T]\times\cdots\times\cV_{N_f}^c[0,T],\\
&\cU_{-i}^d[0,T]=\cU_1^d[0,T]\times\cdots\times\cU_{i-1}^d[0,T]\times\cU_{i+1}^d[0,T]\times\cdots\times\cU_{N_l}^d[0,T],\\
&\cV_{-j}^d[0,T]=\cV_1^d[0,T]\times\cdots\times\cV_{j-1}^d[0,T]\times\cV_{j+1}^d[0,T]\times\cdots\times\cV_{N_f}^d[0,T].
\end{aligned}\right.\end{equation*}
Then any $(u_0,\textbf{u},\textbf{v})\in\cU_0^c[0,T]\times\cU^c[0,T]\times\cV^c[0,T]$ is called an \emph{admissible centralized strategy}, and any $(u_0,\textbf{u},\textbf{v})\in\cU_0^d[0,T]\times\cU^d[0,T]\times\cV^d[0,T]$ is called an \emph{admissible decentralized strategy}.

Let us introduce the following hypothesis on coefficients of state dynamics and cost functionals:
\begin{description}
\item[(H1)] The coefficients of the state equations and cost functionals satisfy the following:
\begin{equation*}\left\{\begin{aligned}
&A_0,A,\tilde{A},C_0,C,\tilde{C},E_0^1,E_0^2,E_1,E_2,F_0^1,F_0^2,F_1,F_2\in\hR^{n\times n};\\
&B_0,D_0\in\hR^{n\times m_1};\q B,D\in\hR^{n\times m_2};\q\tilde{B},\tilde{D}\in\hR^{n\times m_3}.\\
&Q_0,Q,\tilde{Q},H_0,H,\tilde{H}\in\cS^n;\\
&R_0\in\cS^{m_1};\q R\in\cS^{m_2};\q \tilde{R}\in\cS^{m_3}.\\
\end{aligned}\right.\end{equation*}
\item[(H2)] The initial states $\xi_0,\xi_i,\zeta_j\in L^2_{\cF_0}(\Omega;\hR^n)$ are independent; $\hE[\xi_i]=\hE[\zeta_j]=0$, for each $i=1,\ldots,N_l$, $j=1,2,\ldots,N_f$; and there exists $c_0<\infty$ independent of $N_l$ and $N_f$ such that $\sup_{i\geq0}\hE[|\xi_i|^2]\leq c_0$ and $\sup_{j\geq1}\hE[|\zeta_j|^2]\leq c_0$.
\end{description}We point out that no positive-definiteness/non-negativeness conditions on the weighting matrix/matrix-valued functions imposed in \textbf{(H1)}.
Moreover, the coefficients of the convex combination $0\leq\l_0$, $\l$, $\tilde{\l}_1$, $\tilde{\l}_2$, $\tilde{\l}_3\leq 1$. Under \textbf{(H1)}, for any $(x_0,u_0(\cdot))\in\hR^n\times\cU_0^c[0,T]$ (resp., $\hR^n\times\cU_0^d[0,T]$), $(x,u_i(\cdot))\in\hR^n\times\cU_i^c[0,T]$ (resp., $\hR^n\times\cU_i^d[0,T]$), $(y,v_j(\cdot))\in\hR^n\times\cV_j^c[0,T]$ (resp., $\hR^n\times\cV_j^d[0,T]$), \eqref{C6e1}, \eqref{C6e2}, \eqref{C6e3} admits a unique (strong) solution. And the cost functionals \eqref{C6e4}, \eqref{C6e5}, \eqref{C6e6} are also well-defined.

For simplicity, in \textbf{(H2)} it is assumed that all minor leaders and followers have zero initial mean. It is possible to generalize our analysis to deal with different initial means as long as $\{\hE[\xi_i],i\geq1\}$ and $\{\hE[\zeta_j],j\geq1\}$ has a limiting empirical distribution. Now, we can introduce the Stakelberg-Counot-Nash equilibrium as follows.
\begin{definition}
A $(1+N_l+N_f)$-tuple $(\bar{u}_0[\cdot],\bar{\textbf{u}}(\cdot),\bar{\textbf{v}}[\cdot])$, is called an \emph{open-loop Stakelberg-Counot-Nash equilibrium} for the initial states $\xi_0,\xi_i,\zeta_j\in L^2_{\cF_0}(\Omega;\hR^n)$ if:
\begin{equation}\begin{aligned}\label{col}
      &\cJ_j^f(u_0(\cdot),\textbf{u}(\cdot),\bar{v}_j[u_0(\cdot),\textbf{u}(\cdot),\xi_0,\xi_i,\zeta_j](\cdot), \textbf{v}_{-j}(\cdot))\\ &\q=\min_{v_j(\cdot)\in\cV_j^c[0,T]}\cJ_j^f(u_0(\cdot),\textbf{u}(\cdot),v_j(\cdot),\textbf{v}_{-j}(\cdot)),\qq\forall u_0(\cdot)\in\cU_0^c[0,T],\q \textbf{u}(\cdot)\in\cU^c[0,T],\\
      &\cJ_0(\bar{u}_0[\textbf{u}(\cdot),\xi_0,\xi_i,\zeta_j](\cdot),\textbf{u}(\cdot),\bar{\textbf{v}}[\bar{u}_0[\textbf{u}(\cdot),\xi_0,\xi_i,\zeta_j](\cdot), \textbf{u}(\cdot),\xi_0,\xi_i,\zeta_j](\cdot))\\
      &\q=\min_{u_0(\cdot)\in\cU_0^c[0,T]}\cJ_0(u_0(\cdot),\textbf{u}(\cdot),\bar{\textbf{v}}[u_0(\cdot),\textbf{u}(\cdot),\xi_0,\xi_i,\zeta_j](\cdot)),\qq\forall \textbf{u}(\cdot)\in\cU^c[0,T],\\
      &\cJ_i^l(\bar{u}_0[\bar{u}_i(\cdot),\textbf{u}_{-i}(\cdot),\xi_0,\xi_i,\zeta_j](\cdot),\bar{u}_i(\cdot),\textbf{u}_{-i}(\cdot))\\
      &\q=\min_{u_i(\cdot)\in\cU_i^c[0,T]}\cJ_i^l(\bar{u}_0[\textbf{u}(\cdot),\xi_0,\xi_i,\zeta_j](\cdot),u_i(\cdot),\textbf{u}_{-i}(\cdot)),\\
\end{aligned}\end{equation}where $\bar{v}_j:\cU_0^c[0,T]\times\cU^c[0,T]\times L^2_{\cF_0}(\Omega;\hR^n)\times L^2_{\cF_0}(\Omega;\hR^n)\times L^2_{\cF_0}(\Omega;\hR^n)\rightarrow\cV_j^c [0,T]$, and $\bar{u}_0:\cU^c[0,T]\times L^2_{\cF_0}(\Omega;\hR^n)\times L^2_{\cF_0}(\Omega;\hR^n)\times L^2_{\cF_0}(\Omega;\hR^n)\rightarrow\cU_0^c[0,T]$.
\end{definition}If there is no confusion, we use the same notation to denote the optimal respond and the optimal strategy of major leader and minor followers. The above definition of OL strategy is defined in centralized sense. In
particular, game theory has been formulated to capture such individual
interest seeking behavior of the agents in many social,
economic and man-made systems.

For fixed $N=N_{l}+N_{f}$, if each agents can access the full information (states) of other agents, we may view the problem as a standard dynamic LQG leader-follower games and use the full information DPP to derive the Stakelberg-Counot-Nash equilibrium. We now introduce the following definition.
However, scale dynamic model, this approach results in an analytic complexity
which is in general prohibitively high, and correspondingly
leads to few substantive dynamic optimization results. The optimization of large-scale linear control systems
wherein i) many agents are coupled with each other via their
individual dynamics, and ii) the costs are in an individual to
the mass form was presented in  where the theory of
mean field (MF) control (previously termed Nash Certainty
Equivalence) was introduced. It is to be noted that the dynamic
large-scale cost coupled optimization structure of is motivated
by a variety of scenarios, for instance, those analyzed in MFG analysis.

\subsection{Mixed Stakelberg-Counot-Nash equilibrium analysis}

To deal with mixed leader-follower MM dynamic game using MFG theory, one should start with followers. And to deal with a major-minor MFG, one should start with major players. Although the relationships get complicated under our situation, we can still deal with it step by step. That is, firstly, we can solve the optimization problems of followers. The left is a classic major-minor problem and solved in the way of \cite{H10}. The interesting things occur when the major-leader imposes some direct impacts to the followers (i.e., $\tilde{\lambda}_3\neq0$), which will lead to that the state process of major leader will be relied on a kind of forward-backward stochastic differential equation (FBSDE). Generally speaking, it is hard to get the centralized strategy of such mixed Stackelberg MM-MFG. So, let us briefly look at the procedure of finding a decentralized open-loop $\e$-Nash equilibrium strategy of the original problem. And the procedure of finding a decentralized closed-loop $\e$-Nash equilibrium strategy is very similar which will be formulated in next subsection.\\

\textbf{Step 1: MFG analysis of followers}: Let us introduce the auxiliary limiting LQG differential game problems. Firstly, by the Stackelberg game, for given strategy of major leader and minor leaders, followers have to minimize the following cost functionals:
\begin{equation*}\begin{aligned}
\mathcal{J}_j^f(\xi_0,\zeta_j,\bar{m}_X(\cdot);u_0(\cdot),\textbf{u}(\cdot),v_j(\cdot))=&\frac{1}{2}\mathbb{E}\Big\{\int_0^T \Big(\Big\|x_j(t)-\big(\tilde{\lambda}_1 X_0(t)+\tilde{\lambda}_2\bar{m}_X(t)+\tilde{\lambda}_3 x^{(N_f)}(t)\big)\Big\|_{\tilde{Q}}^2\\
&\qq\qq+\|v_j(t)\|_{\tilde{R}}^2\Big)\ud t+\|x_j(T)\|_{\tilde{H}}^2\Big\},
\end{aligned}\end{equation*}
where $\bar{m}_X(\cdot)=\lim_{N_l\rightarrow+\infty}X^{(N_l)}(\cdot)$. Furthermore, as $N_f\rightarrow +\infty,$ we suppose $x^{(N_f)}(\cdot)$ can be approximated by $\mathcal F^0_t$-measurable function $\bar{m}_x(\cdot)$. Then the state process of the follower becomes
\begin{equation}\left\{\begin{aligned}\label{G26}
&\ud \bar{x}_j(t)=\{\tilde{A}\bar{x}_j(t)+\tilde{B}v_j(t)+F_1\bar{m}_x(t)\}\ud t+\{\tilde{C}\bar{x}_j(t)+\tilde{D}v_j(t)+F_2\bar{m}_x(t)\}\ud\widetilde{W}_j(t)\\
&\bar{x}_j(0)=\zeta_j,
\end{aligned}\right.\end{equation}
with the following auxiliary cost functionals
\begin{equation}\begin{aligned}\label{C6e9}
J_j^f(\xi_0,\zeta_j,\bar{m}_X(\cdot),\bar{m}_x(\cdot);u_0(\cdot),v_j(\cdot))=&\frac{1}{2}\mathbb{E}\Big\{\int_0^T \Big(\Big\|x_j(t)-\big(\tilde{\lambda}_1 X_0(t)+\tilde{\lambda}_2\bar{m}_X(t)+\tilde{\lambda}_3 \bar{m}_x(t)\big)\Big\|_{\tilde{Q}}^2\\
&\qq\qq+\|v_j(t)\|_{\tilde{R}}^2\Big)\ud t+\|x_j(T)\|_{\tilde{H}}^2\Big\},
\end{aligned}\end{equation}
for $\cA_j^f,\ 1\leq j \leq N_f$. To distinguish from the original problem, we use the new state variables $\bar{x}_j$ and we will denote $\bar{X}_0$ and $\bar{X}_i$ the new state variables later. But we still use the same set of variables $u_0,u_i,v_j,W_0,W_i,\wt{W}_j$ in this auxiliary limiting problem, and such a reuse of notation should cause no confusion. Then, introduce the following auxiliary Nash game for followers as follows.

\textbf{Problem (OL1).} For given $\xi_0,\zeta_j\in L^2_{\cF_0}(\Omega;\hR^n)$, $\cF^0_t$-measurable functions $\bar{m}_X(\cdot),\bar{m}_x(\cdot)$, and the control $u_0(\cd)$ of major leader $\cA_0$, find an open-loop strategy $\bar{v}_j(\cdot)=\bar{v}_j[u_0(\cdot),\bar{m}_X(\cdot),\bar{m}_x(\cdot),\xi_0,\zeta_j]\in\cV_j^d[0,T]$, $1\leq j\leq N_f$. On other words, Find the Nash equilibrium response functional $\bar{v}_j[\cd]:\cU_0^d[0,T]\times L^2_{\cF^0}(0,T;\hR^n)\times L^2_{\cF^0}(0,T;\hR^n)\times L^2_{\cF_0}(\Omega;\hR^n)\times L^2_{\cF_0}(\Omega;\hR^n)\rightarrow\cV_j^d [0,T]$ of the following Nash differential games among followers:
$$
J_j^f(\xi_0,\zeta_j,\bar{m}_X(\cdot),\bar{m}_x(\cdot);u_0(\cdot),\bar{v}_j[u_0(\cdot),\bar{m}_X(\cdot),\bar{m}_x(\cdot),\xi_0,\zeta_j])= \inf_{v_j(\cdot)\in\cV_j^d[0,T]}J_j^f(\xi_0,\zeta_j,\bar{m}_X(\cdot),\bar{m}_x(\cdot);u_0(\cdot),v_j(\cdot)).
$$The analysis of \textbf{Problem (OL1)} can be further decomposed into substeps using MFG theory.\\

\textbf{Step 1.1} \textbf{(SOC-F)}:
Fixed $\bar{m}_{x}$, and consider the Nash equilibrium response functional of the above \textbf{Problem (OL1)} for representative minor-follower agent denoted by $\bar{v}_j[\cdot]$. For given $\xi_0,\zeta_j\in L^2_{\cF_0}(\Omega;\hR^n)$, $\cF^0_t$-measurable functions $\bar{m}_X(\cdot),\bar{m}_x(\cdot)$, and the control $u_0(\cd)$ of major leader $\cA_0$, find an open-loop strategy $\bar{v}_j(\cdot)=\bar{v}_j[u_0(\cdot),\bar{m}_X(\cdot),\bar{m}_x(\cdot),\xi_0,\zeta_j]\in\cV_j^d[0,T]$, $1\leq j\leq N_f$. On other words, Find the Nash equilibrium response functional $\bar{v}_j[\cd]:\cU_0^d[0,T]\times L^2_{\cF^0}(0,T;\hR^n)\times L^2_{\cF^0}(0,T;\hR^n)\times L^2_{\cF_0}(\Omega;\hR^n)\times L^2_{\cF_0}(\Omega;\hR^n)\rightarrow\cV_j^d [0,T]$ of the following Nash differential games among followers:
$$
J_j^f(\xi_0,\zeta_j,\bar{m}_X(\cdot),\bar{m}_x(\cdot);u_0(\cdot),\bar{v}_j[u_0(\cdot),\bar{m}_X(\cdot),\bar{m}_x(\cdot),\xi_0,\zeta_j])= \inf_{v_j(\cdot)\in\cV_j^d[0,T]}J_j^f(\xi_0,\zeta_j,\bar{m}_X(\cdot),\bar{m}_x(\cdot);u_0(\cdot),v_j(\cdot)).
$$

\textbf{Step 1.2} \textbf{(CC-F)}: applying state-aggregation method, it is possible to determine the state-average limit $\bar{m}_{x}$ by the following condition: $$\hE\Big[\bar{x}_{j}\big(\bar{v}_j[u_0(\cdot),\bar{m}_X(\cdot),\bar{m}_x(\cdot),\xi_0,\zeta_j]\big)\Big|\cF_t^0\Big]=\bar{m}_x.$$By such step, the Nash equilibrium response functional of follower and $\bar{m}_{x}=\bar{m}_{x}(\xi_0,\zeta_j,u_0,\bar{m}_{X})$ can be specified, given any admissible profile announced by leaders.\\

Given the approximate Nash response of all followers, we can turn to the Nash analysis of all leaders. To this, it is necessary to have some MM-MFG analysis when there are both major-minor agents. \\

\textbf{Step 2: MFG analysis of major-leader}: Anticipating the Nash equilibrium response functional of follower $\bar{m}_{x}=\bar{m}_{x}(\xi_0,\zeta_j,u_0,\bar{m}_{X})$, the leaders should solve some Nash equilibrium with size $N_{l}+1$. Similarly, we can assume that as $N_l\rightarrow +\infty,$ we suppose $X^{(N_l)}(\cdot)$ can be approximated by $\mathcal F^0_t$-measurable function $\bar{m}_X(\cdot)$. Then the state process of the major leader and minor leaders becomes
\begin{equation}\left\{\begin{aligned}\label{G27}
&\ud \bar{X}_0(t)=\{A_0\bar{X}_0(t)+B_0u_0(t)+E_0^1\bar{m}_X(t)+F_0^1\bar{m}_x(t)\}\ud t\\
&\qq\qq+\{C_0\bar{X}_0(t)+D_0u_0(t)+E_0^2\bar{m}_X(t)+F_0^2\bar{m}_x(t)\}\ud W_0(t), \\
& \bar{X}_0(0)=\xi_0,
\end{aligned}\right.\end{equation}
and
\begin{equation}\left\{\begin{aligned}\label{G28}
&\ud \bar{X}_i(t)=\{A\bar{X}_i(t)+Bu_i(t)+E_1\bar{m}_X(t)\}\ud t+\{C\bar{X}_i(t)+Du_i(t)+E_2\bar{m}_X(t)\}\ud W_i(t)\\
& \bar{X}_i(0)=\xi_i.
\end{aligned}\right.\end{equation}
with the following auxiliary cost functionals
\begin{equation}\begin{aligned}\label{C6e7}
J_0(\xi_0,\bar{m}_X(\cdot),\bar{m}_x(\cdot);u_0(\cdot))=&\frac{1}{2}\mathbb{E}\Big\{\int_0^T \Big(\Big\|X_0(t)-\big(\lambda_0\bar{m}_X(t)+(1-\lambda_0)\bar{m}_x(t)\big)\Big\|^2_{Q_0}\\
&\qq\qq+\|u_0(t)\|^2_{R_0}\Big)\ud t+\|X_0(T)\|^2_{H_0}\Big\},
\end{aligned}\end{equation}
for $\cA_0$, and
\begin{equation}\begin{aligned}\label{C6e8}
J_i^l(\xi_0,\xi_i,\bar{m}_X(\cdot);u_i(\cdot))=&\frac{1}{2}\mathbb{E}\Big\{\int_0^T \Big(\Big\|X_i(t)-\big(\lambda \bar{m}_X(t)+(1-\lambda)\bar{X}_0(t)\big)\Big\|_{Q}^2\\
&\qq\qq+\|u_i(t)\|_{R}^2\Big)\ud t+\|X_i(T)\|_{H}^2\Big\},
\end{aligned}\end{equation}
for $\cA_i^l,\ 1\leq i \leq N_l$. This can be formulated into MM-MFG. We can analyze the optimal control of major leader first. We can set the following auxiliary problem for the major-leader.

\textbf{Problem (OL2).} For given $\xi_0\in L^2_{\cF_0}(\Omega;\hR^n)$ and $\cF^0_t$-measurable functions $\bar{m}_X(\cdot)$, find an open-loop strategy $\bar{u}_0(\cdot)\in\cU_0^d[0,T]$ such that
$$
J_0(\xi_0,\bar{m}_X(\cdot),\bar{m}_x(\cdot);\bar{u}_0(\cdot))=\inf_{u_0(\cdot)\in\cU_0^d[0,T]}J_0(\xi_0,\bar{m}_X(\cdot),\bar{m}_x(\cdot);u_0(\cdot)).
$$

\textbf{Step 3: MFG analysis of minor-leader}: Anticipating the Nash equilibrium response functional of follower $\bar{m}_{x}=\bar{m}_{x}(\xi_0,\zeta_j,u_0,\bar{m}_{X})$ and the optimal control $\bar{u}_0$ of the major leader. Under the state process \eqref{G28} with the cost functional \eqref{C6e8}, we consider the following problem for minor-leaders.

\textbf{Problem (OL3).} For given $\xi_0,\xi_i\in L^2_{\cF_0}(\Omega;\hR^n)\times L^2_{\cF_0}(\Omega;\hR^n)$, and $\cF^0_t$-measurable functions $\bar{m}_X(\cdot)$, find an open-loop strategy $\bar{u}_i(\cdot)\in\cU_i^d[0,T]$, $1\leq i\leq N_l$, such that
$$
J_i^l(\xi_0,\xi_i,\bar{m}_X(\cdot);\bar{u}_i(\cdot))=\inf_{u_i(\cdot)\in\cU_i^d[0,T]}J_i^l(\xi_0,\xi_i,\bar{m}_X(\cdot);u_i(\cdot)).
$$

\textbf{Step 4: Consistency condition of (Open-loop) Stakelberg-Cornot-Nash equlibrium}: CC condition to determine the frozen $\bar{m}_{X}$ by
$$\hE\Big[\bar{X}_{i}\big(\bar{u}_i(\bar{m}_X)\big)\Big|\cF_t^0\Big]=\bar{m}_X.$$And turn to get its global solvability.

In order to show the steps more clearly, here we illustrate the steps by the figure as follows.

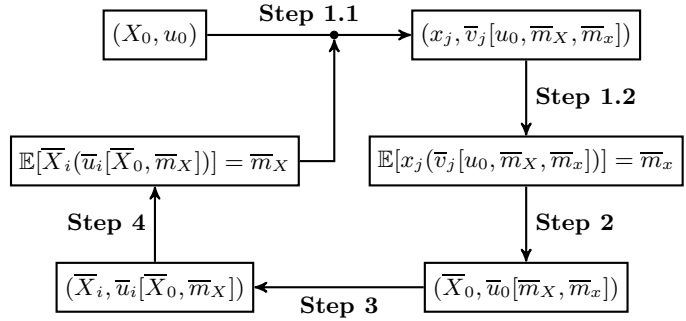
\begin{center}
  \begin{tikzpicture}[scale=1, auto, >=stealth']
    \small
    % node placement with matrix library: 5x4 array
    \matrix[ampersand replacement=\&, row sep=0.2cm, column sep=0.4cm] {
      \node[block] (1) {$(X_0,u_0)$};
      \&
      \node[branch] (p) {};
      \&
      \node[block] (2) {$(x_j,\bar{v}_j[u_0,\bar{m}_X,\bar{m}_x])$};\\\\\\\\\\
      \node[block] (3) {$\mathbb{E}[\bar{X}_{i}(\bar{u}_i[\bar{X}_0,\bar{m}_X])]=\bar{m}_X$};
      \&
      \&
      \node[block] (4) {$\mathbb{E}[x_{j}(\bar{v}_j[u_0,\bar{m}_X,\bar{m}_x])]=\bar{m}_x$};\\\\\\\\\\
      \node[block] (5) {$(\bar{X}_i,\bar{u}_i[\bar{X}_0,\bar{m}_X])$};
      \&
      \&
      \node[block] (6) {$(\bar{X}_0,\bar{u}_0[\bar{m}_X,\bar{m}_x])$};\\\\\\\\\\
    };

    % now link the nodes
    \draw [connector] (1) -- node {\textbf{Step 1.1}} (2);
    \draw [connector] (2) -- node {\textbf{Step 1.2}} (4);
    \draw [connector] (4) -- node {\textbf{Step 2}} (6);
    \draw [connector] (6) -- node {\textbf{Step 3}} (5);
    \draw [connector] (5) -- node {\textbf{Step 4}} (3);
    \draw [connector] (3) -| (p);
  \end{tikzpicture}
\end{center}

\section{Open-loop strategies}From now on, we will suppress time variable $t$ in the equation unless it is necessary. In this section, we study the Mixed S-MM-game strategy in OL sense.

\subsection{Open-loop strategies for the followers}In this subsection, we solve out \textbf{Problem (OL1)} firstly. The main result of this section can be stated as follows.
\begin{theorem}\label{C6l1}
Under assumptions \emph{\textbf{(H1)}}, \emph{\textbf{(H2)}}, and let $\zeta_j\in L^2_{\cF_0}(\Omega;\hR^n)$, $u_0(\cd)\in \cU_0^d[0,T]$, $X_0(\cd)\in L^2_{\cF^0}(0,T;\hR^n)$, $\bar{m}_X(\cd)$, $\bar{m}_x(\cd)\in L^2(0,T;\hR^n)$ be given. Then $\bar{v}_j(\cd)\in\cV_j^d[0,T]$ is an open-loop decentralized optimal control of \emph{\textbf{Problem (OL1)}} for initial value $\zeta_j$ if and only if the following two conditions hold:
\begin{enumerate}
  \item[\emph{(i)}] For $j=1,2,\ldots,N_f$, the adapted solution $(\bar{x}_j(\cd),\bar{y}_j(\cd),\bar{z}_j(\cd))$ to the FBSDE on $[0,T]$
  \begin{equation}\label{C6e11}
    \left\{
    \begin{aligned}
    &\ud\bar{x}_j=\{\tilde{A}\bar{x}_j+\tilde{B}\bar{v}_j+F_1\bar{m}_x\}\ud t+\{\tilde{C}\bar{x}_j+\tilde{D}\bar{v}_j+F_2\bar{m}_x\}\ud\widetilde{W}_j(t)\\
    &\ud\bar{y}_j=-\Big\{\tilde{A}^\top\bar{y}_j+\tilde{C}^\top\bar{z}_j+ \tilde{Q}\Big(\bar{x}_j-\big(\tilde{\lambda}_1X_0+\tilde{\lambda}_2 \bar{m}_X+\tilde{\lambda}_3 \bar{m}_x\big)\Big)\Big\}\ud t+\bar{z}_j\ud\wt{W}_j(t),\\
    &\bar{x}_j(0)=\zeta_j,\ \bar{y}_j(T)=\tilde{H}\bar{x}_j(T),
    \end{aligned}
    \right.
  \end{equation}
  satisfies the following stationarity condition:
  \begin{equation}\label{C6e10}
    \tilde{B}^\top\bar{y}_j+\tilde{R}\bar{v}_j+\tilde{D}^\top\bar{z}_j=0,\qq \mathrm{a.e.}\ t\in[0,T],\ \mathrm{a.s.}
  \end{equation}
  \item[\emph{(ii)}] For $j=1,2,\ldots,N_f$, the following convexity condition holds:
  \begin{equation}\label{b1}
    \hE\Big\{\int_0^T \Big(\Big\langle\tilde{Q}x_j,x_j\Big\rangle+\Big\langle\tilde{R}v_j,v_j\Big\rangle\Big)\ud t+\Big\langle\tilde{H}x_j(T),x_j(T)\Big\rangle\Big\}\geq 0,\q\forall v_j(\cd)\in\cV_j^d[0,T],
  \end{equation}
  where $x_j(\cd)$ is the solution to the FSDE
  \begin{equation}\left\{\begin{aligned}\label{b2}
    &\ud x_j=\Big\{\tilde{A}x_j+\tilde{B}v_j\Big\}\ud t+\Big\{\tilde{C}Y_j+\tilde{D}v_j\Big\}\ud\wt{W}_j(t),\qq t\in[0,T],\\
    &x_j(0)=0.
  \end{aligned}\right.\end{equation}
  Or, equivalently, the map $v_j(\cd)\mapsto J_j^f(\xi_0,\zeta_j,\bar{m}_X(\cdot),\bar{m}_x(\cdot);u_0(\cdot),v_j(\cdot))$, $\forall j=1,2,\ldots,N_f$ is convex.
\end{enumerate}
\end{theorem}

\begin{proof}
For given $\zeta_j\in L^2_{\cF_0}(\Omega;\hR^n)$, $u_0(\cd)\in \cU_0^d[0,T]$, $X_0(\cd)\in L^2_{\cF^0}(0,T;\hR^n)$, $\bar{m}_X(\cd)$, $\bar{m}_x(\cd)\in L^2(0,T;\hR^n)$, and $\bar{v}_j(\cd)\in\cV_j^d[0,T]$, let $(\bar{x}_j(\cd)$, $\bar{y}_j(\cd)$, $\bar{z}_j(\cd))$ be adapted solution to FBSDE \eqref{C6e11}. For any $v_j(\cd)\in\cV_j^d[0,T]$ and $\e\in\hR$, let $x_j^\e(\cd)$ be the solution to the following perturbed state equation on $[0,T]$:
\begin{equation*}\left\{\begin{aligned}
  &\ud x_j^\e=\Big\{\tilde{A}x_j^\e+\tilde{B}(\bar{v}_j+\e v_j)+F_1\bar{m}_x\Big\}\ud t+\Big\{\tilde{C}x_j^\e+\tilde{D}(\bar{v}_j+\e v_j)+F_2\bar{m}_x\Big\}\ud\wt{W}_j(t)\\
  &x_j^\e(0)=\zeta_j.
\end{aligned}\right.\end{equation*}
Then denoting $x_j(\cd)$ the solution to the FSDE \eqref{b2}, we have $x_j^\e(\cd)=\bar{x}_j(\cd)+\e x_j(\cd)$ and
\begin{equation*}\begin{aligned}
&J_j^f(\xi_0,\zeta_j,\bar{m}_X(\cdot),\bar{m}_x(\cdot);u_0(\cdot),\bar{v}_j(\cdot)+\e v_j(\cd))-J_j^f(\xi_0,\zeta_j,\bar{m}_X(\cdot),\bar{m}_x(\cdot);u_0(\cdot), \bar{v}_j(\cdot))\\
=&\frac{\e}{2}\mathbb{E}\Big\{\int_0^T \Big(\Big\langle\tilde{Q}\Big(2\bar{x}_j-2\big(\tilde{\lambda}_1 X_0+\tilde{\lambda}_2 \bar{m}_X+\tilde{\lambda}_3\bar{m}_x\big)+\e x_j\Big),x_j\Big\rangle\\
&\qq+\Big\langle\tilde{R}(2\bar{v}_j+\e v_j),v_j\Big\rangle\Big)\ud t+\Big\langle\tilde{H}(2\bar{x}_j(T)+\e x_j(T)),x_j(T)\Big\rangle\Big\}\\
=&\e\mathbb{E}\Big\{\int_0^T \Big(\Big\langle\tilde{Q}\Big(\bar{x}_j-\big(\tilde{\lambda}_1 X_0+\tilde{\lambda}_2 \bar{m}_X+\tilde{\lambda}_3\bar{m}_x\big)\Big),x_j\Big\rangle+\Big\langle\tilde{R}\bar{v}_j,v_j\Big\rangle\Big)\ud t\\
&+\Big\langle\tilde{H} \bar{x}_j(T),x_j(T)\Big\rangle\Big\}+\frac{\e^2}{2}\hE\Big\{\int_0^T \Big(\Big\langle\tilde{Q}x_j,x_j\Big\rangle+\Big\langle\tilde{R}v_j,v_j\Big\rangle\Big)\ud t+\Big\langle\tilde{H}x_j(T),x_j(T)\Big\rangle\Big\}.\\
\end{aligned}\end{equation*}
On the other hand, applying It\^{o}'s formula to $\Big\langle\bar{y}_j,x_j\Big\rangle$, and taking expectation, we obtain
\begin{equation*}\begin{aligned}
  \hE\Big[\Big\langle\tilde{H}\bar{x}_j(T),x_j(T)\Big\rangle\Big]
  =&\hE\Big\{\int_0^T \Big(\Big\langle\tilde{B}^\top\bar{y}_j+\tilde{D}^\top\bar{z}_j,v_j\Big\rangle\\
  &\q-\Big\langle\tilde{Q}\Big(\bar{x}_j-\big(\tilde{\lambda}_1 X_0+\tilde{\lambda}_2 \bar{m}_X+\tilde{\lambda}_3\bar{m}_x\big)\Big),x_j\Big\rangle\Big)\ud t\Big\}.\\
\end{aligned}\end{equation*}
Hence,
\begin{equation*}\begin{aligned}
&J_j^f(\xi_0,\zeta_j,\bar{m}_X(\cdot),\bar{m}_x(\cdot);u_0(\cdot),\bar{v}_j(\cdot)+\e v_j(\cd))-J_j^f(\xi_0,\zeta_j,\bar{m}_X(\cdot),\bar{m}_x(\cdot);u_0(\cdot), \bar{v}_j(\cdot))\\
=&\e\mathbb{E}\Big\{\int_0^T \Big\langle\tilde{B}^\top\bar{y}_j+\tilde{R}\bar{v}_j+\tilde{D}^\top\bar{z}_j,v_j\Big\rangle\ud t\Big\}\\
&+\frac{\e^2}{2}\hE\Big\{\int_0^T \Big(\Big\langle\tilde{Q}x_j,x_j\Big\rangle+\Big\langle\tilde{R}v_j,v_j\Big\rangle\Big)\ud t+\Big\langle\tilde{H}x_j(T),x_j(T)\Big\rangle\Big\}.\\
\end{aligned}\end{equation*}
It follows that
\begin{equation*}\begin{aligned}
&&J_j^f(\xi_0,\zeta_j,\bar{m}_X(\cdot),\bar{m}_x(\cdot);u_0(\cdot),\bar{v}_j(\cdot))\leq J_j^f(\xi_0,\zeta_j,\bar{m}_X(\cdot),\bar{m}_x(\cdot);u_0(\cdot),\bar{v}_j(\cdot)+\e v_j(\cd)),\\
&&\forall v_j(\cd)\in\cV_j^d[0,T],\ \forall\e\in\hR,
\end{aligned}\end{equation*}
if and only if \eqref{C6e10} and \eqref{b1} hold.
\end{proof}

Furthermore, if we assume that $\tilde{R}$ is invertible, then we have
\begin{equation}\label{C6e12}
\bar{v}_j=-\tilde{R}^{-1}(\tilde{B}^\top\bar{y}_j+\tilde{D}^\top\bar{z}_j),
\end{equation}
so the related Hamiltonian system can be represented by
\begin{equation*}\left\{\begin{aligned}
&\ud\bar{x}_j=\{\tilde{A}\bar{x}_j-\tilde{B}\tilde{R}^{-1}(\tilde{B}^\top\bar{y}_j+\tilde{D}^\top\bar{z}_j)+F_1\bar{m}_x\}\ud t\\
&\qq\qq+\{\tilde{C}\bar{x}_j-\tilde{D}\tilde{R}^{-1}(\tilde{B}^\top\bar{y}_j+\tilde{D}^\top\bar{z}_j)+F_2\bar{m}_x\}\ud\widetilde{W}_j(t)\\
&\ud\bar{y}_j=-\Big\{\tilde{A}^\top\bar{y}_j+\tilde{C}^\top\bar{z}_j+ \tilde{Q}\Big(\bar{x}_j-\big(\tilde{\lambda}_1X_0+\tilde{\lambda}_2 \bar{m}_X+\tilde{\lambda}_3 \bar{m}_x\big)\Big)\Big\}\ud t+\bar{z}_j\ud\wt{W}_j(t),\\
&\bar{x}_j(0)=\zeta_j,\ \bar{y}_j(T)=\tilde{H}\bar{x}_j(T),\ j=1,2,\ldots,N_f,
\end{aligned}\right.\end{equation*}
Based on above analysis, it follows that
\begin{equation}\label{G1}
\bar{m}_x(\cdot)=\lim_{N_f\rightarrow+\infty}\frac{1}{N_f}\sum_{j=1}^{N_f}\bar{x}_j(\cdot)=\hE[\bar{x}_j(\cdot)].
\end{equation}
Here, the first equality of \eqref{G1} is due to the consistency condition: the frozen term $\bar{m}_x(\cdot)$ should equal to the average limit of all realized states $\bar{x}_j(\cdot)$; the second equality is due to the law of large numbers. Thus, by replacing $\bar{m}_x$ by $\hE[\bar{x}_j]$, we get the following system
\begin{equation*}\left\{\begin{aligned}
&\ud\bar{x}_j=\{\tilde{A}\bar{x}_j-\tilde{B}\tilde{R}^{-1}(\tilde{B}^\top\bar{y}_j+\tilde{D}^\top\bar{z}_j)+F_1\hE[\bar{x}_j]\}\ud t\\
&\qq\qq+\{\tilde{C}\bar{x}_j-\tilde{D}\tilde{R}^{-1}(\tilde{B}^\top\bar{y}_j+\tilde{D}^\top\bar{z}_j)+F_2\hE[\bar{x}_j]\}\ud\widetilde{W}_j(t)\\
&\ud\bar{y}_j=-\Big\{\tilde{A}^\top\bar{y}_j+\tilde{C}^\top\bar{z}_j+ \tilde{Q}\Big(\bar{x}_j-\big(\tilde{\lambda}_1X_0+\tilde{\lambda}_2 \bar{m}_X+\tilde{\lambda}_3 \hE[\bar{x}_j]\big)\Big)\Big\}\ud t+\bar{z}_j\ud\wt{W}_j(t),\\
&\bar{x}_j(0)=\zeta_j,\ \bar{y}_j(T)=\tilde{H}\bar{x}_j(T),\ j=1,2,\ldots,N_f,
\end{aligned}\right.\end{equation*}
As all agents are statistically identical, thus we can suppress subscript ``$j$'' and the following consistency condition system arises for generic agent:
\begin{equation}\left\{\begin{aligned}\label{G2}
&\ud\bar{x}=\{\tilde{A}\bar{x}-\tilde{B}\tilde{R}^{-1}(\tilde{B}^\top\bar{y}+\tilde{D}^\top\bar{z})+F_1\hE[\bar{x}]\}\ud t\\
&\qq\qq+\{\tilde{C}\bar{x}-\tilde{D}\tilde{R}^{-1}(\tilde{B}^\top\bar{y}+\tilde{D}^\top\bar{z})+F_2\hE[\bar{x}]\}\ud\widetilde{W}(t)\\
&\ud\bar{y}=-\Big\{\tilde{A}^\top\bar{y}+\tilde{C}^\top\bar{z}+ \tilde{Q}\Big(\bar{x}-\big(\tilde{\lambda}_1X_0+\tilde{\lambda}_2 \bar{m}_X+\tilde{\lambda}_3 \hE[\bar{x}]\big)\Big)\Big\}\ud t+\bar{z}\ud\wt{W}(t),\\
&\bar{x}(0)=\zeta,\ \bar{y}(T)=\tilde{H}\bar{x}(T),
\end{aligned}\right.\end{equation}
where $\wt{W}$ stands for a generic Brownian motion on $(\Omega,\cF,\hP)$ and it is independent of $W_0$. $\zeta$ is a representative element of $\{\zeta_j\}_{1\leq j\leq N_f}$, and $X_0(\cdot)$, $\bar{m}_X(\cdot)$ are to be determined.

\subsection{Open-loop strategies for the major leader}
Once \textbf{Problem (OL1)} is solved, we turn to solve \textbf{Problem (OL2)} about the major leader (agent $\cA_0$). Note that when the followers take their optimal respond $\bar{v}_j(\cdot)$ given by \eqref{C6e12}, the major leader ends up with the following state equation system:
\begin{equation}\label{C6e34}
\left \{ \begin{aligned}
&\ud\bar{X}_0=\{A_0\bar{X}_0+B_0u_0+E_0^1\bar{m}_X+F_0^1\hE[\bar{x}]\}\ud t+\{C_0\bar{X}_0+D_0u_0+E_0^2\bar{m}_X+F_0^2\hE[\bar{x}]\}\ud W_0(t), \\
&\ud\bar{x}=\{\tilde{A}\bar{x}-\tilde{B}\tilde{R}^{-1}(\tilde{B}^\top\bar{y}+\tilde{D}^\top\bar{z})+F_1\hE[\bar{x}]\}\ud t+\{\tilde{C}\bar{x}-\tilde{D}\tilde{R}^{-1}(\tilde{B}^\top\bar{y}+\tilde{D}^\top\bar{z})+F_2\hE[\bar{x}]\}\ud\widetilde{W}(t),\\
&\ud\bar{y}=-\Big\{\tilde{A}^\top\bar{y}+\tilde{C}^\top\bar{z}+ \tilde{Q}\Big(\bar{x}-\big(\tilde{\lambda}_1\bar{X}_0+\tilde{\lambda}_2 \bar{m}_X+\tilde{\lambda}_3 \hE[\bar{x}]\big)\Big)\Big\}\ud t+\bar{z}\ud\wt{W}(t),\\
&\bar{X}_0(0)=\xi_0,\ \bar{x}(0)=\zeta,\ \bar{y}(T)=\tilde{H}\bar{x}(T).
\end{aligned} \right.
\end{equation}
And its cost functional is given by \eqref{C6e7}. Note that equation \eqref{C6e34} is a two-point boundary value problem for SDEs, which is what we call a \emph{forward-backward stochastic differential equation} (FBSDE; see \cite{MY99,Y99,Y02-LQ,Y10}) and the cost functional is still linear quadratic form. Hence, we are going to solve the LQ problem for a FBSDE. Noting that this FBSDE is coupled, therefore, it is not so easy to deal with it. Let us keep in mind that the ``state'' for \eqref{C6e34} is the triple $(\bar{X}_0(\cdot),\bar{x}(\cdot),\bar{y}(\cdot))$. The main result of this section can be stated as follows.

\begin{theorem}\label{C6l2}
Under assumptions \emph{\textbf{(H1)}}, \emph{\textbf{(H2)}}, and let $\xi_0,\zeta\in L^2_{\cF_0}(\Omega;\hR^n)$, $\bar{m}_X(\cd)\in L^2(0,T;\hR^n)$ be given. Then $\bar{u}_0(\cd)\in\cU_0^d[0,T]$ is an open-loop decentralized optimal control of \emph{\textbf{Problem (OL2)}} for initial value $\xi_0$ if and only if the following two conditions hold:
\begin{enumerate}
  \item[\emph{(i)}] The adapted solution $(\bar{X}_0(\cd),\bar{x}(\cd),(\bar{y}(\cd),\bar{z}(\cd)),(Y_0(\cd),Z_0(\cd)),(p(\cd),q(\cd)),K(\cd))$ to the FBSDE on $[0,T]$
  \begin{equation}\label{C6e13}
  \left \{ \begin{aligned}
  &\ud\bar{X}_0=\{A_0\bar{X}_0+B_0\bar{u}_0+E_0^1\bar{m}_X+F_0^1\hE[\bar{x}]\}\ud t+\{C_0\bar{X}_0+D_0\bar{u}_0+E_0^2\bar{m}_X+F_0^2\hE[\bar{x}]\}\ud W_0(t), \\
  &\ud\bar{x}=\{\tilde{A}\bar{x}-\tilde{B}\tilde{R}^{-1}(\tilde{B}^\top\bar{y}+\tilde{D}^\top\bar{z})+F_1\hE[\bar{x}]\}\ud t+\{\tilde{C}\bar{x}-\tilde{D}\tilde{R}^{-1}(\tilde{B}^\top\bar{y}+\tilde{D}^\top\bar{z})+F_2\hE[\bar{x}]\}\ud\widetilde{W}(t),\\
  &\ud\bar{y}=-\Big\{\tilde{A}^\top\bar{y}+\tilde{C}^\top\bar{z}+ \tilde{Q}\Big(\bar{x}-\big(\tilde{\lambda}_1\bar{X}_0+\tilde{\lambda}_2 \bar{m}_X+\tilde{\lambda}_3 \hE[\bar{x}]\big)\Big)\Big\}\ud t+\bar{z}\ud\wt{W}(t),\\
  &\ud Y_0=-\{A_0^\top Y_0+C_0^\top Z_0+Q_0(\bar{X}_0-(\l_0\bar{m}_X+(1-\l_0)\hE[\bar{x}]))+\tilde{Q}\tilde{\l}_1K\}\ud t+Z_0\ud W_0(t),\\
  &\ud p=-\{\tilde{A}^\top p+\tilde{C}^\top q+{F_0^1}^\top\hE[Y_0]+{F_0^2}^\top\hE[Z_0]+F_1^{\top}\hE[p]+F_2^{\top}\hE[q]+\tilde{Q}\wt{\l}_3\hE[K]\\
  &\qq\qq-(1-\l_0)Q_0(\bar{X}_0-(\l_0\bar{m}_X+(1-\l_0)\hE[\bar{x}]))-\tilde{Q}K\}\ud t+q\ud\wt{W}(t),\\
  &\ud K=\{\tilde{A}K+\tilde{B}\tilde{R}^{-1}\tilde{B}^\top p+\tilde{B}\tilde{R}^{-1}\tilde{D}^\top q\}\ud t+\{\tilde{C}K+\tilde{D}\tilde{R}^{-1}\tilde{B}^\top p+\tilde{D}\tilde{R}^{-1}\tilde{D}^\top q\}\ud\wt{W}(t),\\
  &\bar{X}_0(0)=\xi_0,\ \bar{x}(0)=\zeta,\ \bar{y}(T)=\tilde{H}\bar{x}(T),\ Y_0(T)=H_0\bar{X}_0(T),\ p(T)=-\tilde{H}K(T),\ K(0)=0,\\
  \end{aligned} \right.
  \end{equation}
  satisfies the following stationarity condition:
  \begin{equation}\label{C6e24}
  B_0^\top Y_0+D_0^\top Z_0+R_0\bar{u}_{0}=0,\qq \mathrm{a.e.}\ t\in[0,T],\ \mathrm{a.s.}
  \end{equation}
  \item[\emph{(ii)}] The following convexity condition holds:
  \begin{equation}\begin{aligned}\label{C6e27}
    &\mathbb{E}\Big\{\int_0^T \Big(\Big\langle Q_0\Big(X_0-(1-\lambda_0)x\Big),\Big(X_0-(1-\lambda_0)x\Big)\Big\rangle+\Big\langle R_0u_0,u_0\Big\rangle \Big)\ud t\\
    &\qq\qq\qq+\Big\langle H_0X_0(T),X_0(T)\Big\rangle \Big\}\geq 0,\qq\forall u_0(\cd)\in\cU_0^d[0,T],
  \end{aligned}\end{equation}
  where $(X_0(\cd),x(\cd))$ is the solution to the FBSDE
  \begin{equation}\label{C6e20}\left\{\begin{aligned}
  &\ud X_0=\{A_0X_0+B_0u_0+F_0^1\hE[x]\}\ud t+\{C_0X_0+D_0u_0+F_0^2\hE[x]\}\ud W_0(t), \\
  &\ud x=\{\tilde{A}x-\tilde{B}\tilde{R}^{-1}(\tilde{B}^\top y+\tilde{D}^\top z)+F_1\hE[x]\}\ud t+\{\tilde{C}x-\tilde{D}\tilde{R}^{-1}(\tilde{B}^\top y+\tilde{D}^\top z)+F_2\hE[x]\}\ud\widetilde{W}(t),\\
  &\ud y=-\Big\{\tilde{A}^\top y+\tilde{C}^\top z+ \tilde{Q}\Big(x-\big(\tilde{\lambda}_1X_0+\tilde{\lambda}_3 \hE[x]\big)\Big)\Big\}\ud t+z\ud\wt{W}(t),\\
  &X(0)=0,\qq x(0)=0,\qq y(T)=\tilde{H}x(T).
  \end{aligned}\right.\end{equation}
  Or, equivalently, the map $u_0(\cd)\mapsto J_0(\xi_0,\bar{m}_X(\cdot),\bar{m}_x(\cdot);u_0(\cdot))$ is convex.
\end{enumerate}
\end{theorem}

\begin{proof}
For given $\xi_0,\zeta\in L^2_{\cF_0}(\Omega;\hR^n)$, $\bar{m}_X(\cd)\in L^2(0,T;\hR^n)$, and $\bar{u}_0(\cd)\in\cU_0^d[0,T]$, let $(\bar{X}_0(\cd)$, $\bar{x}(\cd)$, $(\bar{y}(\cd),\bar{z}(\cd))$, $(Y_0(\cd),Z_0(\cd))$, $(p(\cd),q(\cd))$, $K(\cd))$ be adapted solution to FBSDE \eqref{C6e13}. For any $u_0(\cd)\in\cU_0^d[0,T]$ and $\e\in\hR$, let $X_0^\e(\cd)$, $x^\e(\cd)$, $(y^\e(\cd),z^e(\cd))$ be the solution to the following perturbed state equation on $[0,T]$:
\begin{equation*}\left\{\begin{aligned}
  &\ud X_0^\e=\{A_0X_0^\e+B_0(\bar{u}_0+\e u_0)+E_0^1\bar{m}_X+F_0^1\hE[x^\e]\}\ud t\\
  &\qq\qq+\{C_0X_0^\e+D_0(\bar{u}_0+\e u_0)+E_0^2\bar{m}_X+F_0^2\hE[x^\e]\}\ud W_0(t), \\
  &\ud x^\e=\{\tilde{A}x^\e-\tilde{B}\tilde{R}^{-1}(\tilde{B}^\top y^\e+\tilde{D}^\top z^\e)+F_1\hE[x^\e]\}\ud t\\
  &\qq\qq+\{\tilde{C}x^\e-\tilde{D}\tilde{R}^{-1}(\tilde{B}^\top y^\e+\tilde{D}^\top z^\e)+F_2\hE[x^\e]\}\ud\widetilde{W}(t),\\
  &\ud y^\e=-\Big\{\tilde{A}^\top y^\e+\tilde{C}^\top z^\e+ \tilde{Q}\Big(x^\e-\big(\tilde{\lambda}_1X_0^\e+\tilde{\lambda}_2 \bar{m}_X+\tilde{\lambda}_3\hE[x^\e]\big)\Big)\Big\}\ud t+z^\e\ud\wt{W}(t),\\
  &X_0^\e(0)=\xi_0,\ x^\e(0)=\zeta,\ y^\e(T)=\tilde{H}x^\e(T).
\end{aligned}\right.\end{equation*}
Then denoting $(X_0(\cd),x(\cd),(y(\cd),z(\cd)))$ the solution to the FBSDE \eqref{C6e20}, we have $X_0^\e(\cd)=\bar{X}_0(\cd)+\e X_0(\cd)$, $x^\e(\cd)=\bar{x}(\cd)+\e x(\cd)$, $y^\e(\cd)=\bar{y}(\cd)+\e y(\cd)$, $z^\e(\cd)=\bar{z}(\cd)+\e z(\cd)$ and
\begin{equation*}\begin{aligned}
&J_0(\xi_0,\bar{m}_X(\cdot),\hE[\bar{x}+\e x];\bar{u}_0(\cdot)+\e u_0(\cdot))-J_0(\xi_0,\bar{m}_X(\cdot),\hE[\bar{x}] ;\bar{u}_0(\cdot))\\
=&\frac{\e}{2}\mathbb{E}\Big\{\int_0^T \Big(\Big\langle Q_0\Big(2\bar{X}_0-2\big(\lambda_0 \bar{m}_X+(1-\lambda_0)\hE[\bar{x}]\big)+\e \Big(X_0-(1-\lambda_0)\hE[x]\Big)\Big),\\
&\Big(X_0-(1-\lambda_0)\hE[x]\Big)\Big\rangle+\Big\langle R_0(2\bar{u}_0+\e u_0),u_0\Big\rangle \Big)\ud t+\Big\langle H_0(2\bar{X}_0(T)+\e X_0(T)),X_0(T)\Big\rangle \Big\}\\
=&\e\mathbb{E}\Big\{\int_0^T \Big(\Big\langle Q_0\Big(\bar{X}_0-\big(\lambda_0 \bar{m}_X+(1-\lambda_0)\hE[\bar{x}]\big)\Big),\Big(X_0-(1-\lambda_0)\hE[x]\Big)\Big\rangle +\Big\langle R_0\bar{u}_0,u_0\Big\rangle \Big)\ud t\\
&+\Big\langle H_0\bar{X}_0(T),X_0(T)\Big\rangle \Big\}+\frac{\e^2}{2}\mathbb{E}\Big\{\int_0^T \Big(\Big\langle Q_0\Big(X_0-(1-\lambda_0)x\Big),\Big(X_0-(1-\lambda_0)x\Big)\Big\rangle \\
&+\Big\langle R_0u_0,u_0\Big\rangle \Big)\ud t+\Big\langle H_0X_0(T),X_0(T)\Big\rangle \Big\}.\\
\end{aligned}\end{equation*}
On the other hand, applying It\^{o}'s formula to $\langle Y_0,X_0\rangle+\langle p,x\rangle+\langle K,y\rangle$, and taking expectation, we obtain
\begin{equation*}\begin{aligned}
\mathbb{E}\Big[H_0\bar{X}_0(T)X_0(T)\Big]&=\mathbb{E}\Big\{\int_0^T \Big(-\Big\langle Q_0\Big(\bar{X}_0-\big(\lambda_0 \bar{m}_X+(1-\lambda_0)\hE[\bar{x}]\big)\Big),\Big(X_0-(1-\lambda_0)\hE[x]\Big)\Big\rangle\\
&\q+\Big\langle B_0^\top Y_0+D_0^\top Z_0, u_0\Big\rangle \Big)\ud t.
\end{aligned}\end{equation*}
Hence,
\begin{equation*}\begin{aligned}
&J_0(\xi_0,\bar{m}_X(\cdot),\hE[\bar{x}+\e x];\bar{u}_0(\cdot)+\e u_0(\cdot))-J_0(\xi_0,\bar{m}_X(\cdot),\hE[\bar{x}] ;\bar{u}_0(\cdot))\\
=&\frac{\e^2}{2}\mathbb{E}\Big\{\int_0^T \Big(\Big\langle Q_0\Big(X_0-(1-\lambda_0)x\Big),\Big(X_0-(1-\lambda_0)x\Big)\Big\rangle+\Big\langle R_0u_0,u_0\Big\rangle \Big)\ud t\\
&+\Big\langle H_0X_0(T),X_0(T)\Big\rangle\Big\}+\e\mathbb{E}\Big\{\int_0^T \Big\langle B_0^\top Y_0+D_0^\top Z_0+R_0\bar{u}_0, u_0\Big\rangle\ud t\Big\}\\
\end{aligned}\end{equation*}
It follows that
\begin{equation*}
J_0(\xi_0,\bar{m}_X(\cdot),\hE[\bar{x}] ;\bar{u}_0(\cdot))\leq J_0(\xi_0,\bar{m}_X(\cdot),\hE[\bar{x}+\e x];\bar{u}_0(\cdot)+\e u_0(\cdot)),\ \forall u_0(\cd)\in\cU_0^d[0,T],\ \forall\e\in\hR,
\end{equation*}
if and only if \eqref{C6e24} and \eqref{C6e27} hold.
\end{proof}

Similarly, if we assume $R_0$ is invertible, then we can represent the optimal control by
\begin{equation}\label{C6e28}
\bar{u}_0=-R_0^{-1}(B_0^\top Y_0+D_0^\top Z_0).
\end{equation}
Then the following coupled system follows
\begin{equation}\label{C6e29}
\left\{ \begin{aligned}
  &\ud\bar{X}_0=\{A_0\bar{X}_0-B_0R_0^{-1}(B_0^\top Y_0+D_0^\top Z_0)+E_0^1\bar{m}_X+F_0^1\hE[\bar{x}]\}\ud t\\
  &\qq\qq+\{C_0\bar{X}_0-D_0R_0^{-1}(B_0^\top Y_0+D_0^\top Z_0)+E_0^2\bar{m}_X+F_0^2\hE[\bar{x}]\}\ud W_0(t), \\
  &\ud\bar{x}=\{\tilde{A}\bar{x}-\tilde{B}\tilde{R}^{-1}(\tilde{B}^\top\bar{y}+\tilde{D}^\top\bar{z})+F_1\hE[\bar{x}]\}\ud t+\{\tilde{C}\bar{x}-\tilde{D}\tilde{R}^{-1}(\tilde{B}^\top\bar{y}+\tilde{D}^\top\bar{z})+F_2\hE[\bar{x}]\}\ud\widetilde{W}(t)\\
  &\ud K=\{\tilde{A}K+\tilde{B}\tilde{R}^{-1}\tilde{B}^\top p+\tilde{B}\tilde{R}^{-1}\tilde{D}^\top q\}\ud t+\{\tilde{C}K+\tilde{D}\tilde{R}^{-1}\tilde{B}^\top p+\tilde{D}\tilde{R}^{-1}\tilde{D}^\top q\}\ud\wt{W}(t)\\
  &\ud Y_0=-\{A_0^\top Y_0+C_0^\top Z_0+Q_0(\bar{X}_0-(\l_0\bar{m}_X+(1-\l_0)\hE[\bar{x}]))+\tilde{Q}\tilde{\l}_1K\}\ud t+Z_0\ud W_0(t),\\
  &\ud\bar{y}=-\Big\{\tilde{A}^\top\bar{y}+\tilde{C}^\top\bar{z}+ \tilde{Q}\Big(\bar{x}-\big(\tilde{\lambda}_1\bar{X}_0+\tilde{\lambda}_2 \bar{m}_X+\tilde{\lambda}_3 \hE[\bar{x}]\big)\Big)\Big\}\ud t+\bar{z}\ud\wt{W}(t),\\
  &\ud p=-\{\tilde{A}^\top p+\tilde{C}^\top q+{F_0^1}^\top\hE[Y_0]+{F_0^2}^\top\hE[Z_0]+F_1^{\top}\hE[p]+F_2^{\top}\hE[q]+\tilde{Q}\wt{\l}_3\hE[K]\\
  &\qq\qq-(1-\l_0)Q_0(\bar{X}_0-(\l_0\bar{m}_X+(1-\l_0)\hE[\bar{x}]))-\tilde{Q}K\}\ud t+q\ud\wt{W}(t),\\
  &\bar{X}_0(0)=\xi_0,\ \bar{x}(0)=\zeta,\ K(0)=0,\ Y_0(T)=H_0\bar{X}_0(T),\ \bar{y}(T)=\tilde{H}\bar{x}(T),\ p(T)=-\tilde{H}K(T),\\
\end{aligned}\right.
\end{equation}
where $\bar{m}_X(\cdot)$ is to be determined.

\subsection{Open-loop strategies for the minor leaders}
Once \textbf{Problem (OL2)} is solved, we turn to solve \textbf{Problem (OL3)} about the minor leaders (agents $\cA_i^l$, $1\leq i\leq N_l$). Note that when the followers takes their optimal responds $\bar{v}_j(\cdot)$ given by \eqref{C6e12}, and the major leader takes his optimal control $\bar{u}_0(\cdot)$ given by \eqref{C6e28}, the minor leaders ends up with the following state equation system:
\begin{equation*}\left\{\begin{aligned}
&\ud \bar{X}_i=\{A\bar{X}_i+Bu_i+E_1\bar{m}_X\}\ud t+\{C\bar{X}_i+Du_i+E_2\bar{m}_X\}\ud W_i(t)\\
&\bar{X}_i(0)=\xi_i,\ i=1,2,\ldots,N_l.
\end{aligned}\right.\end{equation*}
And its cost functional is given by \eqref{C6e8} with $\bar{X}_0(\cdot)$ being from \eqref{C6e29}. So it is similar to solve \textbf{Problem (OL1)}, and the main result in this section can be stated as follows.
\begin{theorem}\label{C6l3}
Under assumptions \emph{\textbf{(H1)}}, \emph{\textbf{(H2)}}, and let $\xi_0,\xi_i\in L^2_{\cF_0}(\Omega;\hR^n)$, $\bar{u}_0(\cd)\in \cU_0^d[0,T]$, $\bar{m}_X(\cd)\in L^2(0,T;\hR^n)$ be given. Then $\bar{u}_i(\cd)\in\cU_i^d[0,T]$ is a decentralized optimal control of \emph{\textbf{Problem (OL3)}} for initial value $\xi_i$ if and only if the following two conditions hold:
\begin{enumerate}
  \item[\emph{(i)}] For $i=1,2,\ldots,N_l$, the adapted solution $(\bar{X}_i(\cd),\bar{Y}_i(\cd),\bar{Z}_i(\cd))$ to the FBSDE on $[0,T]$
  \begin{equation}\label{C6e23}
  \left \{ \begin{aligned}
    &\ud\bar{X}_i=\{A\bar{X}_i+B\bar{u}_i+E_1\bar{m}_X\}\ud t+\{C\bar{X}_i+D\bar{u}_i+E_2\bar{m}_X\}\ud W_i(t)\\
    &\ud\bar{Y}_i=-\Big\{A^\top\bar{Y}_i+C^\top\bar{Z}_i+Q\Big(\bar{X}_i-\big(\lambda\bar{m}_X+(1-\lambda)\bar{X}_0 \big)\Big)\Big\}\ud t+\bar{Z}_i\ud W_i(t),\\
    &\bar{X}_i(0)=\xi_i,\ \bar{Y}_i(T)=H\bar{X}_i(T),
  \end{aligned} \right.\end{equation}
  satisfies the following stationarity condition:
  \begin{equation}\label{C6e22}
    B^\top \bar{Y}_i+R\bar{u}_i+D^\top \bar{Z}_i=0,\qq \mathrm{a.e.}\ t\in[0,T],\ \mathrm{a.s.}
  \end{equation}
  \item[\emph{(ii)}] For $i=1,2,\ldots,N_l$, the following convexity condition holds:
  \begin{equation}\label{b5}
    \mathbb{E}\Big\{\int_0^T \Big(\Big\langle QX_i,X_i\Big\rangle+\Big\langle Ru_i,u_i\Big\rangle\Big)\ud t+\Big\langle HX_i(T),X_i(T)\Big\rangle\Big\}\geq 0,\q\forall u_i(\cd)\in\cU_i^d[0,T],
  \end{equation}
  where $X_i(\cd)$ is the solution to the FSDE
  \begin{equation}\left\{\begin{aligned}\label{b6}
    &\ud X_i=\Big\{AX_i+Bu_i\Big\}\ud t+\Big\{CX_i+Du_i\Big\}\ud W_i(t), \qq t\in[0,T],\\
    &X_i(0)=0.
  \end{aligned}\right.\end{equation}
  Or, equivalently, the map $u_i(\cd)\mapsto J_i^l(\xi_0,\xi_i,\bar{m}_X(\cdot);\bar{u}_0(\cdot),u_i(\cdot))$ is convex (for $i=1,2,\ldots,N_l$).
\end{enumerate}
\end{theorem}

\begin{proof}
For given $\xi_0,\xi_i\in L^2_{\cF_0}(\Omega;\hR^n)$, $\bar{u}_0(\cd)\in\cU_0^d[0,T]$, $\bar{m}_X(\cd)\in L^2(0,T;\hR^n)$, and $\bar{u}_i(\cd)\in\cU_i^d[0,T]$, let $(\bar{X}_i(\cd)$, $\bar{Y}_i(\cd)$, $\bar{Z}_i(\cd))$ be adapted solution to FBSDE \eqref{C6e23}. For any $u_i(\cd)\in\cU_i^d[0,T]$ and $\e\in\hR$, let $X_i^\e(\cd)$ be the solution to the following perturbed state equation on $[0,T]$:
\begin{equation*}\left\{\begin{aligned}
  &\ud X_i^\e=\Big\{AX_i^\e+B(\bar{u}_i+\e u_i)+E_1\bar{m}_X\Big\}\ud t+\Big\{CX_i^\e+D(\bar{u}_i+\e u_i)+E_2\bar{m}_X\Big\}\ud W_i(t),\\
  &X_i^\e(0)=x.
\end{aligned}\right.\end{equation*}
Then denoting $X_i(\cd)$ the solution to the FSDE \eqref{b6}, we have $X_i^\e(\cd)=\bar{X}_i(\cd)+\e X_i(\cd)$ and
\begin{equation*}\begin{aligned}
&J_i^l(\xi_0,\xi_i,\bar{m}_X(\cdot);\bar{u}_0(\cdot),\bar{u}_i(\cdot)+\e u_i(\cdot))-J_i^l(\xi_0,\xi_i,\bar{m}_X(\cdot);\bar{u}_0(\cdot),\bar{u}_i(\cdot))\\
=&\frac{\e}{2}\mathbb{E}\Big\{\int_0^T \Big(\Big\langle Q\Big(2\bar{X}_i-2\big(\l\bar{m}_X+(1-\l)\bar{X}_0\big)+\e X_i\Big),X_i\Big\rangle+\Big\langle R(2\bar{u}_i+\e u_i),u_i\Big\rangle\Big)\ud t\\
&+\Big\langle H(2\bar{X}_i(T)+\e X_i(T)),X_i(T)\Big\rangle \Big\}\\
=&\e\mathbb{E}\Big\{\int_0^T \Big(\Big\langle Q\Big(\bar{X}_i-\big(\l\bar{m}_X+(1-\l)\bar{X}_0\big)\Big),X_i\Big\rangle +\Big\langle R \bar{u}_i,u_i\Big\rangle \Big)\ud t+\Big\langle H \bar{X}_i(T),X_i(T)\Big\rangle \Big\}\\
&+\frac{\e^2}{2}\mathbb{E}\Big\{\int_0^T \Big(\Big\langle QX_i,X_i\Big\rangle+\Big\langle Ru_i,u_i\Big\rangle\Big)\ud t+\Big\langle HX_i(T),X_i(T)\Big\rangle\Big\}.\\
\end{aligned}\end{equation*}
On the other hand, applying It\^{o}'s formula to $\Big\langle\bar{Y}_i,X_i\Big\rangle$, and taking expectation, we obtain
\begin{equation*}\begin{aligned}
\hE\Big[\Big\langle H\bar{X}_i(T),X_i(T)\Big\rangle \Big]
=&\hE\Big\{\int_0^T \Big(\Big\langle B^\top \bar{Y}_i+D^\top \bar{Z}_i,u_i\Big\rangle-\Big\langle Q\Big(\bar{X}_i-\big(\l\bar{m}_X+(1-\l)\bar{X}_0\big)\Big),X_i\Big\rangle \Big)\ud t\Big\}.\\
\end{aligned}\end{equation*}
Hence,
\begin{equation*}\begin{aligned}
&J_i^l(\xi_0,\xi_i,\bar{m}_X(\cdot);\bar{u}_0(\cdot),\bar{u}_i(\cdot)+\e u_i(\cdot))-J_i^l(\xi_0,\xi_i,\bar{m}_X(\cdot);\bar{u}_0(\cdot),\bar{u}_i(\cdot))\\
=&\e\mathbb{E}\Big\{\int_0^T\Big\langle B^\top\bar{Y}_i+R\bar{u}_i+D^\top\bar{Z}_i,u_i\Big\rangle \ud t\Big\}+\frac{\e^2}{2}\mathbb{E}\Big\{\int_0^T \Big(\Big\langle QX_i,X_i\Big\rangle+\Big\langle Ru_i,u_i\Big\rangle\Big)\ud t+\Big\langle HX_i(T),X_i(T)\Big\rangle\Big\}.\\
\end{aligned}\end{equation*}
It follows that
\begin{equation*}
J_i^l(\xi_0,\xi_i,\bar{m}_X(\cdot);\bar{u}_0(\cdot),\bar{u}_i(\cdot))\leq J_i^l(\xi_0,\xi_i,\bar{m}_X(\cdot);\bar{u}_0(\cdot),\bar{u}_i(\cdot)+\e u_i(\cdot)),\qq\forall u_i(\cd)\in\cU_i^d[0,T],\ \forall\e\in\hR,
\end{equation*}
if and only if \eqref{C6e22} and \eqref{b5} hold.
\end{proof}

Furthermore, if we assume that $R$ is invertible, then we have
\begin{equation}\label{C6e32}
\bar{u}_i=-R^{-1}(B^\top \bar{Y}_i+D^\top \bar{Z}_i),
\end{equation}
so the related Hamiltonian system can be represented by
\begin{equation*}\left\{\begin{aligned}
&\ud\bar{X}_i=\{A\bar{X}_i-BR^{-1}(B^\top \bar{Y}_i+D^\top \bar{Z}_i)+E_1\bar{m}_X\}\ud t\\
&\qq\qq+\{C\bar{X}_i-DR^{-1}(B^\top \bar{Y}_i+D^\top \bar{Z}_i)+E_2\bar{m}_X\}\ud W_i(t)\\
&\ud\bar{Y}_i=-\Big\{A^\top\bar{Y}_i+C^\top\bar{Z}_i+Q\Big(\bar{X}_i-\big(\lambda\bar{m}_X+(1-\lambda)\bar{X}_0\big)\Big)\Big\}\ud t+\bar{Z}_i\ud W_i(t),\\
&\bar{X}_i(0)=\xi_i,\ \bar{Y}_i(T)=H\bar{X}_i(T),
\end{aligned}\right.\end{equation*}
Based on the above analysis, it follows that
\begin{equation}\label{G3}
\bar{m}_X(\cdot)=\lim_{N_l\rightarrow+\infty}\frac{1}{N_l}\sum_{i=1}^{N_l}\bar{X}_i(\cdot)=\hE[\bar{X}_i(\cdot)].
\end{equation}
Here, the first equality of \eqref{G3} is due to the consistency condition: the frozen term $\bar{m}_X(\cdot)$ should equal to the average limit of all realized states $\bar{X}_i(\cdot)$; the second equality is due to the law of large numbers. Thus, by replacing $\bar{m}_X$ by $\hE[\bar{X}_i]$, we get the following system
\begin{equation*}\left\{\begin{aligned}
&\ud\bar{X}_i=\{A\bar{X}_i-BR^{-1}(B^\top \bar{Y}_i+D^\top \bar{Z}_i)+E_1\hE[\bar{X}_i]\}\ud t+\{C\bar{X}_i-DR^{-1}(B^\top \bar{Y}_i+D^\top \bar{Z}_i)+E_2\hE[\bar{X}_i]\}\ud W_i(t)\\
&\ud\bar{Y}_i=-\Big\{A^\top\bar{Y}_i+C^\top\bar{Z}_i+Q\Big(\bar{X}_i-\big(\lambda\hE[\bar{X}_i] +(1-\lambda)\bar{X}_0\big)\Big)\Big\}\ud t+\bar{Z}_i\ud W_i(t),\\
&\bar{X}_i(0)=\xi_i,\ \bar{Y}_i(T)=H\bar{X}_i(T),
\end{aligned}\right.\end{equation*}
As all agents are statistically identical, thus we can suppress subscript ``$i$'' and the following consistency condition system arises for generic agent:
\begin{equation}\left\{\begin{aligned}\label{G4}
&\ud\bar{X}=\{A\bar{X}-BR^{-1}(B^\top \bar{Y}+D^\top \bar{Z})+E_1\hE[\bar{X}]\}\ud t+\{C\bar{X}-DR^{-1}(B^\top \bar{Y}+D^\top \bar{Z})+E_2\hE[\bar{X}]\}\ud W(t)\\
&\ud\bar{Y}=-\Big\{A^\top\bar{Y}+C^\top\bar{Z}+Q\Big(\bar{X}-\big(\lambda\hE[\bar{X}]+(1-\lambda)\bar{X}_0\big)\Big)\Big\}\ud t+\bar{Z}\ud W(t),\\
&\bar{X}(0)=\xi,\ \bar{Y}(T)=H\bar{X}(T),
\end{aligned}\right.\end{equation}
where $W$ stands for a generic Brownian motion on $(\Omega,\cF,\hP)$, and it is independent of $W_0,\wt{W}$. $\xi$ is a representative element of $\{\xi_i\}_{1\leq i\leq N_l}$.

To the end of the section, combined with \eqref{C6e29} and \eqref{G4}, replacing $\bar{m}_X$ by $\hE[\bar{X}]$, we can get the consistency condition system for open-loop strategy as follows.
\begin{equation}\left\{\begin{aligned}\label{G5}
  &\ud\bar{X}_0=\{A_0\bar{X}_0-B_0R_0^{-1}(B_0^\top Y_0+D_0^\top Z_0)+E_0^1\hE[\bar{X}]+F_0^1\hE[\bar{x}]\}\ud t\\
  &\qq\qq+\{C_0\bar{X}_0-D_0R_0^{-1}(B_0^\top Y_0+D_0^\top Z_0)+E_0^2\hE[\bar{X}]+F_0^2\hE[\bar{x}]\}\ud W_0(t), \\
  &\ud\bar{X}=\{A\bar{X}-BR^{-1}(B^\top \bar{Y}+D^\top \bar{Z})+E_1\hE[\bar{X}]\}\ud t+\{C\bar{X}-DR^{-1}(B^\top \bar{Y}+D^\top \bar{Z})+E_2\hE[\bar{X}]\}\ud W(t)\\
  &\ud\bar{x}=\{\tilde{A}\bar{x}-\tilde{B}\tilde{R}^{-1}(\tilde{B}^\top\bar{y}+\tilde{D}^\top\bar{z})+F_1\hE[\bar{x}]\}\ud t+\{\tilde{C}\bar{x}-\tilde{D}\tilde{R}^{-1}(\tilde{B}^\top\bar{y}+\tilde{D}^\top\bar{z})+F_2\hE[\bar{x}]\}\ud\widetilde{W}(t)\\
  &\ud K=\{\tilde{A}K+\tilde{B}\tilde{R}^{-1}\tilde{B}^\top p+\tilde{B}\tilde{R}^{-1}\tilde{D}^\top q\}\ud t+\{\tilde{C}K+\tilde{D}\tilde{R}^{-1}\tilde{B}^\top p+\tilde{D}\tilde{R}^{-1}\tilde{D}^\top q\}\ud\wt{W}(t)\\
  &\ud Y_0=-\{A_0^\top Y_0+C_0^\top Z_0+Q_0(\bar{X}_0-(\l_0\hE[\bar{X}]+(1-\l_0)\hE[\bar{x}]))+\tilde{Q}\tilde{\l}_1K\}\ud t+Z_0\ud W_0(t),\\
  &\ud\bar{Y}=-\Big\{A^\top\bar{Y}+C^\top\bar{Z}+Q\Big(\bar{X}-\big(\lambda\hE[\bar{X}]+(1-\lambda)\bar{X}_0\big)\Big)\Big\}\ud t+\bar{Z}\ud W(t),\\
  &\ud\bar{y}=-\Big\{\tilde{A}^\top\bar{y}+\tilde{C}^\top\bar{z}+ \tilde{Q}\Big(\bar{x}-\big(\tilde{\lambda}_1\bar{X}_0+\tilde{\lambda}_2 \hE[\bar{X}]+\tilde{\lambda}_3 \hE[\bar{x}]\big)\Big)\Big\}\ud t+\bar{z}\ud\wt{W}(t),\\
  &\ud p=-\{\tilde{A}^\top p+\tilde{C}^\top q+{F_0^1}^\top\hE[Y_0]+{F_0^2}^\top\hE[Z_0]+F_1^{\top}\hE[p]+F_2^{\top}\hE[q]+\tilde{Q}\wt{\l}_3\hE[K]\\
  &\qq\qq-(1-\l_0)Q_0(\bar{X}_0-(\l_0\hE[\bar{X}]+(1-\l_0)\hE[\bar{x}]))-\tilde{Q}K\}\ud t+q\ud\wt{W}(t),\\
  &\bar{X}_0(0)=\xi_0,\q \bar{X}(0)=\xi,\q \bar{x}(0)=\zeta,\q K(0)=0,\\
  &Y_0(T)=H_0\bar{X}_0(T),\q \bar{Y}(T)=H\bar{X}(T),\q \bar{y}(T)=\tilde{H}\bar{x}(T),\q p(T)=-\tilde{H}K(T),\\
\end{aligned}\right.\end{equation}

\section{The Consistency Condition System}
Under assumptions \textbf{(H1)}, \textbf{(H2)}, when $\tilde{R}(\cdot)$, $R_0(\cdot)$ and $R(\cdot)$ are always invertible, we get the consistency condition (CC) for OL strategy in section 3. In this section, we turn to verify the well-posedness of the CC equation.

For the simplicity of notation, denote
$\mathbb{X}^\top=(\bar{X}_0,\bar{X},\bar{x},K)$,
$\mathbb{Y}^\top=(Y_0,\bar{Y},\bar{y},p)$,
$\mathbb{Z}^\top=(Z_0,\bar{Z},\bar{z},q)$,
$\textbf{W}^\top=(W_0,W,\tilde{W},\tilde{W})$, and then the consistency condition system \eqref{G5} can be rewritten as
\begin{equation}\label{G16}
\left \{ \begin{aligned}
\ud \mathbb{X}=&\{\textbf{A}\mathbb{X}+\bar{\textbf{A}}\hE[\mathbb{X}]+\textbf{B}\mathbb{Y}+\textbf{E}\mathbb{Z}\}\ud t+\{\textbf{C}\mathbb{X}+\bar{\textbf{C}}\hE[\mathbb{X}]+\textbf{D}\mathbb{Y}+\textbf{F}\mathbb{Z}\}\circ\ud \textbf{W}(t)\\
\ud \mathbb{Y}=&-\{\textbf{A}^\top\mathbb{Y}+\textbf{A}_0^\top\hE[\mathbb{Y}]+\textbf{C}^\top\mathbb{Z}+\textbf{C}_0^\top\hE[\mathbb{Z}] +\textbf{Q}\mathbb{X}+\bar{\textbf{Q}}\hE[\mathbb{X}]\}\ud t+\mathbb{Z}\circ\ud \textbf{W}(t),\\
\mathbb{X}(0)=&\textbf{X}_0,\qq \mathbb{Y}(T)=\textbf{H}_0 \mathbb{X}(T),
\end{aligned} \right.
\end{equation}
where
\begin{equation*}\begin{aligned}
&\textbf{A}=\begin{pmatrix}\begin{smallmatrix}A_0&0&0&0\\ 0&A&0&0\\ 0&0&\tilde{A}&0\\ 0&0&0&\tilde{A}\end{smallmatrix}\end{pmatrix},\qq\bar{\textbf{A}}=\begin{pmatrix}\begin{smallmatrix}0&E_0^1&F_0^1&0\\ 0&E_1&0&0\\ 0&0&F_1&0\\ 0&0&0&0\end{smallmatrix}\end{pmatrix},&&\textbf{C}=\begin{pmatrix}\begin{smallmatrix}C_0&0&0&0\\ 0&C&0&0\\ 0&0&\tilde{C}&0\\ 0&0&0&\tilde{C}\end{smallmatrix}\end{pmatrix},\qq\bar{\textbf{C}}=\begin{pmatrix}\begin{smallmatrix}0&E_0^2&F_0^2&0\\ 0&E_2&0&0\\ 0&0&F_2&0\\ 0&0&0&0\end{smallmatrix}\end{pmatrix},\\
&\textbf{B}=\begin{pmatrix}\begin{smallmatrix}-B_0R_0^{-1}B_0^\top&0&0&0\\ 0&-BR^{-1}B^\top&0&0\\ 0&0&-\tilde{B}\tilde{R}^{-1}\tilde{B}^\top&0\\ 0&0&0&\tilde{B}\tilde{R}^{-1}\tilde{B}^\top\end{smallmatrix}\end{pmatrix},&&\textbf{D}=\begin{pmatrix}\begin{smallmatrix}-D_0R_0^{-1}B_0^\top&0&0&0\\ 0&-DR^{-1}B^\top&0&0\\ 0&0&-\tilde{D}\tilde{R}^{-1}\tilde{B}^\top&0\\0&0&0&\tilde{D}\tilde{R}^{-1}\tilde{B}^\top\end{smallmatrix}\end{pmatrix},\\
&\textbf{E}=\begin{pmatrix}\begin{smallmatrix}-B_0R_0^{-1}D_0^\top&0&0&0\\ 0&-BR^{-1}D^\top&0&0\\ 0&0&-\tilde{B}\tilde{R}^{-1}\tilde{D}^\top&0\\ 0&0&0&\tilde{B}\tilde{R}^{-1}\tilde{D}^\top\end{smallmatrix}\end{pmatrix}, &&\textbf{F}=\begin{pmatrix}\begin{smallmatrix}-D_0R_0^{-1}D_0^\top&0&0&0\\0&-DR^{-1}D^\top&0&0\\0&0&-\tilde{D}\tilde{R}^{-1}\tilde{D}^\top&0\\ 0&0&0&\tilde{D}\tilde{R}^{-1}\tilde{D}^\top\end{smallmatrix}\end{pmatrix},\\
&\textbf{A}_0=\begin{pmatrix}\begin{smallmatrix}0&0&0&{F_0^1}\\0&0&0&0\\ 0&0&0&0\\0&0&0&F_1\end{smallmatrix}\end{pmatrix},\qq\textbf{C}_0=\begin{pmatrix}\begin{smallmatrix}0&0&0&{F_0^2}\\ 0&0&0&0\\ 0&0&0&0\\0&0&0&F_2\end{smallmatrix}\end{pmatrix},&&\textbf{H}_0=\begin{pmatrix}\begin{smallmatrix}H_0&0&0&0\\ 0&H&0&0\\ 0&0&\tilde{H}&0\\ 0&0&0&-\tilde{H}\end{smallmatrix}\end{pmatrix},\qq\textbf{X}_0=\begin{pmatrix}\begin{smallmatrix}\xi_0\\ \xi\\ \zeta\\0\end{smallmatrix}\end{pmatrix},\\
&\textbf{Q}=\begin{pmatrix}\begin{smallmatrix}-Q_0&0&0&-\tilde{Q}\tilde{\lambda}_1\\ Q(1-\l)&-Q&0&0\\ \tilde{Q}\tilde{\lambda}_1&0&-\tilde{Q}&0\\ Q_0(1-\lambda_0)&0&0&\tilde{Q}\end{smallmatrix}\end{pmatrix},
&&\bar{\textbf{Q}}=\begin{pmatrix}\begin{smallmatrix}0&Q_0\lambda_0&Q_0(1-\lambda_0)&0\\ 0&Q\lambda&0&0\\ 0&\tilde{Q}\tilde{\lambda}_2&\tilde{Q}\tilde{\lambda}_3&0\\ 0&-Q_0\lambda_0(1-\lambda_0)&-Q_0(1-\lambda_0)^2&-\tilde{Q}\tilde{\lambda}_3\end{smallmatrix}\end{pmatrix}, \\
\end{aligned}\end{equation*}

\subsection{Decoupling for open-loop strategy}
Then, we turn to decouple the FBSDE \eqref{G16} by Riccati equation. Note that
$$\ud\hE[\mathbb{X}]=\Big[(\textbf{A}+\bar{\textbf{A}})\hE[\mathbb{X}]+\textbf{B}\hE[\mathbb{Y}]+\textbf{E}\hE[\mathbb{Z}]\Big]\ud t.$$
Hence,
\begin{equation*}\begin{aligned}
\ud\Big(\mathbb{X}-\hE[\mathbb{X}]\Big)=&\Big[\textbf{A}(\mathbb{X}-\hE[\mathbb{X}])+\textbf{B}(\mathbb{Y}-\hE[\mathbb{Y}]) +\textbf{E}(\mathbb{Z}-\hE[\mathbb{Z}])\Big]\ud t\\
&+[\textbf{C}(\mathbb{X}-\hE[\mathbb{X}])+(\textbf{C}+\bar{\textbf{C}})\hE[\mathbb{X}]+\textbf{D}\mathbb{Y}+\textbf{F}\mathbb{Z}]\circ\ud \textbf{W}(t).
\end{aligned}\end{equation*}
Now, we assume that
\begin{equation}\label{C6e14}
\mathbb{Y}(t)=P(t)\Big(\mathbb{X}(t)-\hE[\mathbb{X}(t)]\Big)+\Pi(t)\hE[\mathbb{X}(t)],\qq t\in[0,T],
\end{equation}
for some deterministic and differentiable functions $P(\cd)$ and $\Pi(\cd)$, taking values in $\cS^{4n}$, such that
$$P(T)=\textbf{H}_0,\qq\Pi(T)=\textbf{H}_0.$$ Then $$\hE[\mathbb{Y}(t)]=\Pi(t)\hE[\mathbb{X}(t)]$$ and $$\mathbb{Y}(t)-\hE[\mathbb{Y}(t)]=P(t) \Big(\mathbb{X}(t)-\hE[\mathbb{X}(t)]\Big).$$ Therefore,
\begin{equation}\begin{aligned}\label{C6e15}
\ud\mathbb{Y}&=\Big[\dot{P}(\mathbb{X}-\hE[\mathbb{X}])+\dot{\Pi}\hE[\mathbb{X}]\Big]\ud t+P\cd\ud\Big(\mathbb{X}-\hE[\mathbb{X}]\Big)+\Pi\cd\ud\hE[\mathbb{X}]\\
&=\bigg\{\dot{P}(\mathbb{X}-\hE[\mathbb{X}])+\dot{\Pi}\hE[\mathbb{X}]+P\Big[\textbf{A}(\mathbb{X}-\hE[\mathbb{X}]) +\textbf{B}(\mathbb{Y}-\hE[\mathbb{Y}])+\textbf{E}(\mathbb{Z}-\hE[\mathbb{Z}])\Big]\\
&\q+\Pi\Big[(\textbf{A}+\bar{\textbf{A}}) \hE[\mathbb{X}]+\textbf{B}\hE[\mathbb{Y}]+\textbf{E}\hE[\mathbb{Z}]\Big]\bigg\}\ud t\\
&\q+P\Big[\textbf{C}(\mathbb{X}-\hE[\mathbb{X}])+(\textbf{C}+\bar{\textbf{C}})\hE[\mathbb{X}]+\textbf{D}\mathbb{Y}+\textbf{F}\mathbb{Z}\Big]\circ\ud\textbf{W}(t).
\end{aligned}\end{equation}
Comparing the diffusion terms, we should have
\begin{equation}\begin{aligned}\label{C6e16}
\mathbb{Z}=(I-P\textbf{F})^{-1}P\Big[\textbf{C}(\mathbb{X}-\hE[\mathbb{X}])+(\textbf{C}+\bar{\textbf{C}})\hE[\mathbb{X}]+\textbf{D}\mathbb{Y}\Big].
\end{aligned}\end{equation}
Then $$\hE[\mathbb{Z}]=(I-P\textbf{F})^{-1}P\Big[(\textbf{C}+\bar{\textbf{C}})\hE[\mathbb{X}]+\textbf{D}\hE[\mathbb{Y}]\Big],$$ and $$\mathbb{Z}-\hE[\mathbb{Z}]=(I-P\textbf{F})^{-1}P\Big[\textbf{C}(\mathbb{X}-\hE[\mathbb{X}])+\textbf{D}(\mathbb{Y}-\hE[\mathbb{Y}])\Big].$$
Comparing the drift terms, we should have
\begin{equation}\begin{aligned}\label{C6e17}
0=&\dot{P}(\mathbb{X}-\hE[\mathbb{X}])+\dot{\Pi}\hE[\mathbb{X}]+P\Big[\textbf{A}(\mathbb{X}-\hE[\mathbb{X}]) +\textbf{B}(\mathbb{Y}-\hE[\mathbb{Y}])+\textbf{E}(\mathbb{Z}-\hE[\mathbb{Z}])\Big]\\
&+\Pi\Big[(\textbf{A}+\bar{\textbf{A}}) \hE[\mathbb{X}]+\textbf{B}\hE[\mathbb{Y}]+\textbf{E}\hE[\mathbb{Z}]\Big]\\ &+[\textbf{A}^\top\mathbb{Y}+\textbf{A}_0^\top\hE[\mathbb{Y}]+\textbf{C}^\top\mathbb{Z}+\textbf{C}_0^\top\hE[\mathbb{Z}] +\textbf{Q}\mathbb{X}+\bar{\textbf{Q}}\hE[\mathbb{X}]]\\
\ =&\bigg\{\dot{P}+P\textbf{A}+\textbf{A}^\top P+\textbf{C}^\top(I-P\textbf{F})^{-1}P\textbf{C}+\textbf{Q}+P\textbf{B}P\\
&+P\textbf{E}(I-P\textbf{F})^{-1}P(\textbf{C}+\textbf{D}P)\bigg\}\Big(\mathbb{X}-\hE[\mathbb{X}]\Big)\\
&+\bigg\{\dot{\Pi}+\Pi(\textbf{A}+\bar{\textbf{A}})+(\textbf{A}^\top+\textbf{A}_0^\top)\Pi+\Pi\textbf{B}\Pi+\textbf{Q}+\bar{\textbf{Q}}\\
&+(\Pi\textbf{E}+\textbf{C}^\top+\textbf{C}_0^\top)(I-P\textbf{F})^{-1}P(\textbf{C}+\bar{\textbf{C}}+\textbf{D}\Pi)\bigg\}\hE[\mathbb{X}].
\end{aligned}\end{equation}
Therefore, we should let $P(\cd)$ and $\Pi(\cd)$ be the solutions to the following Riccati equations, respectively:
\begin{equation}\left\{\begin{aligned}\label{C6e18}
&\dot{P}+P\textbf{A}+\textbf{A}^\top P+\textbf{C}^\top(I-P\textbf{F})^{-1}P\textbf{C}+\textbf{Q}+P\textbf{B}P\\
&\qq+P\textbf{E}(I-P\textbf{F})^{-1}P(\textbf{C}+\textbf{D}P)=0,\\
&P(T)=\textbf{H}_0,
\end{aligned}\right.\end{equation}
and
\begin{equation}\left\{\begin{aligned}\label{C6e19}
&\dot{\Pi}+\Pi(\textbf{A}+\bar{\textbf{A}})+(\textbf{A}^\top+\textbf{A}_0^\top)\Pi+\Pi\textbf{B}\Pi+\textbf{Q}+\bar{\textbf{Q}}\\
&\qq+(\Pi\textbf{E}+\textbf{C}^\top+\textbf{C}_0^\top)(I-P\textbf{F})^{-1}P(\textbf{C}+\bar{\textbf{C}}+\textbf{D}\Pi)=0,\\
&\Pi(T)=\textbf{H}_0,
\end{aligned}\right.\end{equation}

\subsection{Decoupling for the feedback strategy}
Except the pure open-loop method, we can also introducing the following Riccati equations to decouple the Hamiltonian systems first.

The Hamiltonian system of minor follower is
\begin{equation*}\left\{\begin{aligned}
&\ud\bar{x}_j=[\tilde{A}\bar{x}_j+\tilde{B}\bar{v}_j+F_1\bar{m}_x]\ud t+[\tilde{C}\bar{x}_j+\tilde{D}\bar{v}_j+F_2\bar{m}_x]\ud \wt{W}_j\\
&\ud\bar{y}_j=-[\tilde{A}^\top\bar{y}_j+\tilde{C}^\top\bar{z}_j+\tilde{Q}(\bar{x}_j-(\tilde{\l}_1X_0+\tilde{\l}_2\bar{m}_X+\tilde{\l}_3\bar{m}_x))]\ud t +\bar{z}_j\ud\wt{W}_j\\
&\bar{x}_j(0)=\zeta_j,\bar{y}_j(T)=\tilde{H}\bar{x}_j,
\end{aligned}\right.\end{equation*}
with the stationary condition
\begin{equation*}
\bar{v}_j=-\tilde{R}^{-1}(\tilde{B}^\top\bar{y}_j+\tilde{D}^\top\bar{z}_j).
\end{equation*}
Assume that $\bar y_j=P_1\bar x_j+\Phi_1$, and we can get the Riccati equations
\begin{equation*}\left\{\begin{aligned}
&\dot{P}_1+\tilde{A}^\top P_1+P_1\tilde{A}+\tilde{Q}-P_1\tilde{B}\tilde{R}^{-1}\tilde{B}^\top P_1+\widetilde{\cS}^\top\widetilde{\cR}^{-1}P_1\widetilde{\cS}=0\\
&P_1(T)=\tilde{H},
\end{aligned}\right.\end{equation*}
and
\begin{equation*}\left\{\begin{aligned}
&\ud\Phi_1=\Big(-\tilde{A}^\top \Phi_1+P_1\tilde{B}\tilde{R}^{-1}\tilde{B}^\top\Phi_1+\tilde{Q}(\tilde{\l}_1X_0 +\tilde{\l}_2\bar{m}_X+\tilde{\l}_3\bar{m}_x)-P_1F_1\bar{m}_x-\widetilde{\cS}^\top\widetilde{\cR}^{-1}\widetilde{f}\Big)\ud t\\
&\Phi_1(T)=0,
\end{aligned}\right.\end{equation*}
where
\begin{equation*}\left\{\begin{aligned}
&\widetilde{\cR}:=I+P_1\tilde{D}\tilde{R}^{-1}\tilde{D}^\top,\qq\widetilde{\cS}:=\tilde{C}-\tilde{D}\tilde{R}^{-1} \tilde{B}^\top P_1,\\
&\widetilde{f}:=P_1F_2\bar{m}_x -P_1\tilde{D}\tilde{R}^{-1}\tilde{B}^\top\Phi_1.
\end{aligned}\right.\end{equation*}
Note that
\begin{equation*}\textbf{minor follower:}\left\{\begin{aligned}
&\bar{y}_j=P_1\bar{x}_j+\Phi_1\\
&\bar{z}_j=\widetilde{\cR}^{-1}P_1\widetilde{\cS}\bar{x}_j+\widetilde{\cR}^{-1}\widetilde{f},
\end{aligned}\right.\end{equation*}
so the feedback is
\begin{equation*}\begin{aligned}
&\bar{v}_j=-\tilde{R}^{-1}\Big(\tilde{B}^\top P_1+\tilde{D}^\top\widetilde{\cR}^{-1}P_1\widetilde{\cS}\Big)\bar{x}_j- \tilde{R}^{-1}\tilde{B}^\top\Phi_1-\tilde{R}^{-1}\tilde{D}^\top\widetilde{\cR}^{-1}\widetilde{f}.
\end{aligned}\end{equation*}
The major leader ends up with the following Hamiltonian system
\begin{equation*}\left\{\begin{aligned}
&\ud \bar{X}_0=(A_0\bar{X}_0+B_0\bar{u}_0+E_0^1\bar{m}_X+F_0^1\bar{m}_x)\ud t+(C_0\bar{X}_0+D_0\bar{u}_0+E_0^2\bar{m}_X+F_0^2\bar{m}_x)\ud W_0,\\
&\ud \bar{m}_x=(\widetilde{\cA}\bar{m}_x+\widetilde{\cB}\Phi_1)\ud t+(\widetilde{\cC}\bar{m}_x+\widetilde{\cD}\Phi_1)\ud W_0,\\
&\ud \Phi_1=(\tilde{Q}\tilde{\l}_1\bar{X}_0+\hat{\cQ}\bar{m}_x +\hat{\cA} \Phi_1+\tilde{Q}\tilde{\l}_2\bar{m}_X)\ud t,\\
&\ud Y_0=-\bigg\{A_0^\top Y_0+C_0^\top Z_0+Q_0\Big(\bar{X}_0-(\l_0\bar{m}_X+(1-\l_0)\bar{m}_x)\Big) +\tilde{Q}\tilde{\l}_1y_2\bigg\}\ud t+Z_0\ud W_0,\\
&\ud y_1=-\bigg\{\widetilde{\cA}^\top y_1+\widetilde{\cC}^\top z_1+\hat{\cQ}^\top y_2+{F_0^1}^\top Y_0+{F_0^2}^\top Z_0-(1-\l_0)Q_0\Big(\bar{X}_0-(\l_0\bar{m}_X+(1-\l_0)\bar{m}_x)\Big)\bigg\}\ud t+z_1\ud W_0,\\
&\ud y_2=-(\widetilde{\cB}^\top y_1+\widetilde{\cD}^\top z_1+\hat{\cA}^\top y_2)\ud t,\\
&\bar{X}_0(0)=\xi_0,\bar{m}_x(0)=0,\Phi_1(T)=0,Y_0(T)=H_0\bar{X}_0(T),y_1(T)=0,y_2(0)=0,
\end{aligned}\right.\end{equation*}
with the stationary condition
\begin{equation*}
\bar{u}_0=-R_0^{-1}(B_0^\top Y_0+D_0^\top Z_0),
\end{equation*}where
\begin{equation*}\left\{\begin{aligned}
&\widetilde{\cA}:=\tilde{A}+F_1-\tilde{B}\tilde{R}^{-1}\big(\tilde{B}^\top P_1+\tilde{D}^\top \widetilde{\cR}^{-1}(P_1\widetilde{\cS}+P_1F_2)\big),&&\widetilde{\cB}:=(\tilde{B}\tilde{R}^{-1}\tilde{D}^\top\widetilde{\cR}^{-1}P_1\tilde{D}-\tilde{B})\tilde{R}^{-1} \tilde{B}^\top,\\
&\widetilde{\cC}:=\tilde{C}+F_2-\tilde{D}\tilde{R}^{-1}\big(\tilde{B}^\top P_1+\tilde{D}^\top \widetilde{\cR}^{-1}(P_1\widetilde{\cS}+P_1F_2)\big),&&\widetilde{\cD}:=(\tilde{D}\tilde{R}^{-1}\tilde{D}^\top\widetilde{\cR}^{-1}P_1\tilde{D}-\tilde{D}) \tilde{R}^{-1}\tilde{B}^\top,\\
&\hat{\cA}:=(\widetilde{\cS}^\top\widetilde{\cR}^{-1}P_1\tilde{D}+P_1\tilde{B})\tilde{R}^{-1}\tilde{B}^\top-\tilde{A}^\top, &&\hat{\cQ}:=\tilde{Q}\tilde{\l}_3-P_1F_1-\widetilde{\cS}^\top\widetilde{\cR}^{-1}P_1F_2.\\
\end{aligned}\right.\end{equation*}
Assume that $\rbjt{Y}=P_2\rbjt{X}+\Phi_2$, where
\begin{equation*}
\rbjt{X}=\left(\begin{smallmatrix}
           X_0 \\
           \bar{m}_x \\
           \Phi_1
         \end{smallmatrix}\right),\qquad
\rbjt{Y}=\left(\begin{smallmatrix}
           Y_0 \\
           y_1 \\
           y_2
         \end{smallmatrix}\right),\qquad
\rbjt{Z}=\left(\begin{smallmatrix}
           Z_0 \\
           z_1 \\
           0
         \end{smallmatrix}\right),\qquad
\rbjt{W}_0=\left(\begin{smallmatrix}
           W_0 \\
           W_0 \\
           0
         \end{smallmatrix}\right),
\end{equation*}and for simplicity, we rewrite the Hamiltonian system by
\begin{equation*}\left\{\begin{aligned}
&\ud\rbjt{X}=(L_{11}\rbjt{X}+L_{12}\rbjt{Y}+L_{13}\rbjt{Z}+f_1)\ud t+(L_{21}\rbjt{X}+L_{22}\rbjt{Y}+L_{23}\rbjt{Z}+f_2)\circ\ud \rbjt{W}_0,\\
&\ud\rbjt{Y}=(L_{31}\rbjt{X}+L_{32}\rbjt{Y}+L_{33}\rbjt{Z}+f_3)\ud t+\rbjt{Z}\circ\ud \rbjt{W}_0,\\
\end{aligned}\right.\end{equation*}
where
\begin{equation*}\begin{aligned}
&L_{11}=\begin{pmatrix}\begin{smallmatrix}A_0&F_0^1&0\\0&\tilde{\cA}&\tilde{\cB}\\ \tilde{Q}\tilde{\l}_1&\hat{\cQ}& \hat{\cA}\end{smallmatrix}\end{pmatrix},
&&L_{12}=\begin{pmatrix}\begin{smallmatrix}-B_0R_0^{-1}B_0^\top&0&0\\ 0&0&0\\ 0&0&0\end{smallmatrix}\end{pmatrix},
&&L_{13}=\begin{pmatrix}\begin{smallmatrix}-B_0R_0^{-1}D_0^\top&0&0\\ 0&0&0\\ 0&0&0\end{smallmatrix}\end{pmatrix},\\
&L_{21}=\begin{pmatrix}\begin{smallmatrix}C_0&F_0^2&0\\0&\tilde{\cC}&\tilde{\cD}\\0&0&0\end{smallmatrix}\end{pmatrix},
&&L_{22}=\begin{pmatrix}\begin{smallmatrix}-D_0R_0^{-1}B_0^\top&0&0\\ 0&0&0\\ 0&0&0\end{smallmatrix}\end{pmatrix},
&&L_{23}=\begin{pmatrix}\begin{smallmatrix}-D_0R_0^{-1}D_0^\top&0&0\\ 0&0&0\\ 0&0&0\end{smallmatrix}\end{pmatrix},\\
&L_{31}=\begin{pmatrix}\begin{smallmatrix}-Q_0&Q_0(1-\l_0)&0\\Q_0(1-\l_0)&-Q_0(1-\l_0)^2&0\\ 0&0&0\end{smallmatrix}\end{pmatrix},
&&L_{32}=\begin{pmatrix}\begin{smallmatrix}-A_0^\top&0&-\tilde{Q}\tilde{\l}_1\\-{F_0^1}^\top&-\tilde{\cA}^\top&-\hat{\cQ}^\top\\ 0&-\tilde{\cB}^\top&-\hat{\cA}^\top\end{smallmatrix}\end{pmatrix},
&&L_{33}=\begin{pmatrix}\begin{smallmatrix}-C_0^\top&0&0\\-{F_0^2}^\top&-\tilde{\cC}^\top&0\\ 0&-\tilde{\cD}^\top&0\end{smallmatrix}\end{pmatrix},\\
&f_{1}=\begin{pmatrix}\begin{smallmatrix}E_0^1\bar{m}_X\\ 0\\ \tilde{Q}\tilde{\l}_2\bar{m}_X\end{smallmatrix}\end{pmatrix},
&&f_{2}=\begin{pmatrix}\begin{smallmatrix}E_0^2\bar{m}_X\\ 0\\ 0\end{smallmatrix}\end{pmatrix},
&&f_{3}=\begin{pmatrix}\begin{smallmatrix}Q_0\l_0\bar{m}_X\\ -Q_0(1-\l_0)^2\bar{m}_X\\ 0\end{smallmatrix}\end{pmatrix},\\
\end{aligned}\end{equation*}
then we can get the following Riccati equations
\begin{equation*}\left\{\begin{aligned}
&\dot{P}_2+P_2L_{11}-L_{32}P_2-L_{31}+P_2L_{12}P_2\\
&\qq+(P_2L_{13}-L_{33})(I-P_2L_{23})^{-1}P_2(L_{21}+L_{22}P_2)=0\\
&P_2(T)=H_0,
\end{aligned}\right.\end{equation*}
and
\begin{equation*}\left\{\begin{aligned}
&\dot{\Phi}_2+\Big((P_2L_{12}-L_{32})+(P_2L_{13}-L_{33})(I-P_2L_{23})^{-1}P_2L_{22}\Big)\Phi_2\\
&\qq+P_2f_1+(P_2L_{13}-L_{33})(I-P_2L_{23})^{-1}f_2-f_3=0\\
&\Phi_2(T)=0.
\end{aligned}\right.\end{equation*}
So the feedback is
\begin{equation*}\begin{aligned}
\bar{u}_0=&-\left(\begin{smallmatrix}
                    R_0^{-1} & 0 & 0 \\
                    0 & 0 & 0 \\
                    0 & 0 & 0
                  \end{smallmatrix}\right)\left[
                  \left(\begin{smallmatrix}
                    B_0^\top & 0 & 0 \\
                    0 & 0 & 0 \\
                    0 & 0 & 0
                  \end{smallmatrix}\right)(P_2\rbjt{X}+\Phi_2)+
                  \left(\begin{smallmatrix}
                    D_0^\top & 0 & 0 \\
                    0 & 0 & 0 \\
                    0 & 0 & 0
                  \end{smallmatrix}\right)(I-P_2L_{23})^{-1}P_2[(L_{21}+L_{22}P_2)\rbjt{X}+L_{22}\Phi_2]\right]\\
         =&-\left(\begin{smallmatrix}
                    R_0^{-1} & 0 & 0 \\
                    0 & 0 & 0 \\
                    0 & 0 & 0
                  \end{smallmatrix}\right)\left[
                  \left(\begin{smallmatrix}
                    B_0^\top & 0 & 0 \\
                    0 & 0 & 0 \\
                    0 & 0 & 0
                  \end{smallmatrix}\right)\left(P_2+\left(\begin{smallmatrix}
                    D_0^\top & 0 & 0 \\
                    0 & 0 & 0 \\
                    0 & 0 & 0
                  \end{smallmatrix}\right)(I-P_2L_{23})^{-1}P_2(L_{21}+L_{22}P_2)\right)\rbjt{X}\right.\\
         &\left.+\left(\left(\begin{smallmatrix}
                    B_0^\top & 0 & 0 \\
                    0 & 0 & 0 \\
                    0 & 0 & 0
                  \end{smallmatrix}\right)+\left(\begin{smallmatrix}
                    D_0^\top & 0 & 0 \\
                    0 & 0 & 0 \\
                    0 & 0 & 0
                  \end{smallmatrix}\right)(I-P_2L_{23})^{-1}P_2L_{22}\right)\Phi_2\right]\\
\end{aligned}\end{equation*}
At last, the Hamiltonian system of minor leader is
\begin{equation*}\left\{\begin{aligned}
&\ud\bar{X}_i=[A\bar{X}_j+B\bar{u}_i+E_1\bar{m}_X]\ud t+[C\bar{X}_i+D\bar{u}_i+E_2\bar{m}_X]\ud W_i\\
&\ud\bar{Y}_i=-[A^\top\bar{Y}_i+C^\top\bar{Z}_i+Q(\bar{X}_i-(\l\bar{m}_X+(1-\l)\bar{X}_0))]\ud t +\bar{Z}_i\ud W_i\\
&\bar{X}_i(0)=\xi_i,\bar{Y}_i(T)=H\bar{X}_i,
\end{aligned}\right.\end{equation*}
with the stationary condition
\begin{equation*}
\bar{u}_i=-R^{-1}(B^\top\bar{Y}_i+D^\top\bar{Z}_i).
\end{equation*}
Assume that $\bar{Y}_i=P_3\bar{X}_i+\Phi_3$, and we can get the Riccati equations
\begin{equation*}\left\{\begin{aligned}
&\dot{P}_3+A^\top P_3+P_3A+Q-P_3BR^{-1}B^\top P_3+\cS^\top\cR^{-1}P_1\cS=0\\
&P_3(T)=H,
\end{aligned}\right.\end{equation*}
and
\begin{equation*}\left\{\begin{aligned}
&\ud\Phi_3=\Big(-A^\top \Phi_3+P_3BR^{-1}B^\top\Phi_3+Q(\l\bar{m}_X+(1-\l)\bar{X}_0)-P_3E_1\bar{m}_X-\cS^\top\cR^{-1}f\Big)\ud t\\
&\Phi_3(T)=0,
\end{aligned}\right.\end{equation*}
where
\begin{equation*}\left\{\begin{aligned}
&\cR:=I+P_3DR^{-1}D^\top,\qq \cS:=C-DR^{-1}B^\top P_3,\\
&f:=P_3E_2\bar{m}_X -P_3DR^{-1}B^\top\Phi_3.
\end{aligned}\right.\end{equation*}
Note that
\begin{equation*}\textbf{minor leader:}\left\{\begin{aligned}
&\bar{Y}_i=P_3\bar{X}_i+\Phi_3\\
&\bar{Z}_i=\cR^{-1}P_3\cS\bar{X}_i+\cR^{-1}f,
\end{aligned}\right.\end{equation*}
so the feedback is
\begin{equation*}\begin{aligned}
&\bar{u}_i=-R^{-1}\Big(B^\top P_3+D^\top\cR^{-1}P_3\cS\Big)\bar{X}_i-R^{-1}B^\top\Phi_3-R^{-1}D^\top\cR^{-1}f.
\end{aligned}\end{equation*}

\section{$\e$-Nash Equilibrium Analysis}
In above sections, we obtained the decentralized open-loop strategy of the mixed S-MM-MFG through the consistency condition system. Now we turn to verify that it is the SNC approximate equilibrium (i.e. $\e$-Stackelberg-Nash-Cournot equilibrium). In order to ensure the solvability of the open-loop strategy, we assume the Riccati equation \eqref{C6e18} and \eqref{C6e19} admits a unique solution. At the beginning, we first present the definition of $\e$-SNC equilibrium.
\begin{definition}
  A set of controls $(\bar{u}_0,\bar{u}_1,\ldots,\bar{u}_{N_l},\bar{v}_1,\ldots,\bar{v}_{N_f})\in\cU_0^d\times\cU^d\times\cV^d$, for $(1+N_l+N_f)$ agents is called to satisfy an $\e$-SNC equilibrium with respect to the costs $(\cJ_0,\cJ_1^l,\ldots,\cJ_{N_l}^l,\cJ_1^f,\cJ_{N_f}^f)$, if there exists $\e=\e(N)\geq 0$ ($N=\min\{N_l,N_f\}$), $\lim_{N\rightarrow\infty}\e(N)=0$ such that for any fixed $i=1,2,\ldots,N_l$, $j=1,2,\ldots,N_f$, we have
  \begin{equation}\label{C6e33}\left\{\begin{aligned}
    &\cJ_0(\bar{u}_0,\bar{u},\bar{v})\leq\cJ_0(u_0,\bar{u},\bar{v})+\e,\\
    &\cJ_i^l(\bar{u}_0,\bar{u}_i,\bar{u}_{-i})\leq\cJ_i^l(\bar{u}_0,u_i,\bar{u}_{-i})+\e,\\
    &\cJ_j^f(\bar{u}_0,\bar{u},\bar{v}_j,\bar{v}_{-j})\leq\cJ_j^f(\bar{u}_0,\bar{u},v_j,\bar{v}_{-j})+\e,
  \end{aligned}\right.\end{equation}
  when any alternative control $(u_0,u_i,v_j)\in\cU_0^d\times\cU_i^d\times\cV_j^d$ is applied by $(\cA_0,\cA_i,\cB_j)$.
\end{definition}
At first, we present the main result in this section and its proof will be given later.
\begin{theorem}\label{thm5.1}
Under assumptions {\rm(H1)-(H2)}, and if the conditions in the Theorem 3.1, Theorem 3.2, Theorem 3.3 hold, then $(\bar{u}_0,\bar{u}_i,\bar{v}_j)$ is an $\e$-Nash equilibrium of mixed S-MM-MFG for major leader agent $\cA_0$, each minor leader agent $\cA_i^l$, $i=1,2,\ldots,N_l$, and each follower agent $\cA_j^f$, $j=1,2,\ldots,N_f$. And $(\bar{u}_0,\bar{u}_i,\bar{v}_j)$ is given by
\begin{equation}\label{C6t2}\left\{\begin{aligned}
\bar{u}_0(t)=&-R_0^{-1}(t)(B_0(t)^\top Y_0(t)+D_0(t)^\top Z_0(t)),\\
\bar{u}_i(t)=&-R^{-1}(t)(B(t)^\top \bar{Y}_i(t)+D(t)^\top \bar{Z}_i(t)),\\
\bar{v}_j(t)=&-\tilde{R}^{-1}(t)(\tilde{B}(t)^\top \bar{y}_j(t)+\tilde{D}(t)^\top \bar{z}_j(t)),
\end{aligned}\right.\end{equation}
for $(Y_0,Z_0)$, $(\bar{Y}_i,\bar{Z}_i)$, $(\bar{y}_j,\bar{z}_j)$ solved by \eqref{G5}.
\end{theorem}

For major leader $\cA_0$, minor leaders $\cA_i^l$ and followers $\cA_j^f$, the decentralized states $\bar{X}_0$, $\bar{X}_i$ and $\bar{Y}_j$ are given respectively by
\begin{equation}\left\{\begin{aligned}\label{N1}
\ud \bar{X}_0=&[A_0\bar{X}_0-B_0 R_0^{-1}(B_0^\top Y_0+D_0^\top Z_0)+E_0^1 \bar{X}^{(N_l)}+F_0^1 \bar{x}^{(N_f)}]\ud t\\
&+[C_0\bar{X}_0-D_0 R_0^{-1}(B_0^\top Y_0+D_0^\top Z_0)+E_0^2 \bar{X}^{(N_l)}+F_0^2 \bar{x}^{(N_f)}]\ud W_0, \\
\ud \bar{X}_i=&[A\bar{X}_i-BR^{-1}(B^\top \bar{Y}_i+D^\top\bar{Z}_i)+E_1 \bar{X}^{(N_l)}]\ud t+[C\bar{X}_i-DR^{-1}(B^\top \bar{Y}_i+D^\top\bar{Z}_i)+E_2 \bar{X}^{(N_l)}]\ud W_i\\
\ud\bar{x}_j=&[\tilde{A}\bar{x}_j-\tilde{B}\tilde{R}^{-1}(\tilde{B}^\top\bar{y}_j+\tilde{D}^\top\bar{z}_j)+F_1\bar{x}^{(N_f)}]\ud t+[\tilde{C}\bar{x}_j-\tilde{D}\tilde{R}^{-1}(\tilde{B}^\top\bar{y}_j+\tilde{D}^\top\bar{z}_j)+F_2\bar{x}^{(N_f)}]\ud\widetilde{W}_j\\
\bar{X}_0(0)=&\xi_0,\qq \bar{X}_i(0)=\xi_i,\qq \bar{x}_j(0)=\z_j,
\end{aligned}\right.\end{equation}
where the processes $(Y_0,Z_0)$, $(\bar{Y}_i,\bar{Z}_i)$, $(\bar{y}_j,\bar{z}_j)$ solved by \eqref{G5}. Let us first present following several lemmas and for the simplicity of notation, we denote the inner product $\langle\ \cd\ ,\ \cd\ \rangle=|\ \cd\ |^2$.

\begin{lemma}\label{l5.1}
Under assumptions {\rm(H1)-(H2)}, and if the conditions in the Theorem \ref{C6l1}, Theorem \ref{C6l2}, Theorem \ref{C6l3} hold, then there exists a constant $M$ independent of $N_l$ and $N_f$, such that
$$\sup_{0\leq i\leq N_l}\hE\Big[\sup_{0\leq t\leq T}\big|\bar{X}_i(t)\big|^2\Big]<M,$$
$$\sup_{1\leq j\leq N_f}\hE\Big[\sup_{0\leq t\leq T}\big|\bar{x}_j(t)\big|^2\Big]<M.$$
\end{lemma}
\begin{proof}
From Theorem \ref{C6l1}, Theorem \ref{C6l2}, Theorem \ref{C6l3}, the FBSDEs \eqref{C6e11}, \eqref{C6e13} and \eqref{C6e23} have a unique solution $(\bar{X}_0,Y_0,Z_0)\in L^2_{\cF^0}(0,T;\hR^{3n})$, $(\bar{X}_i,\bar{Y}_i,\bar{Z}_i)\in L^2_{\cF^i}(0,T;\hR^{3n})$ and $(\bar{x}_j,\bar{y}_j,\bar{z}_j)\in L^2_{\cG^j}(0,T;\hR^{3n})$, $1\leq i\leq N_l$, $1\leq j\leq N_f$. Thus, SDEs system \eqref{N1} has also a unique solution
$$(\bar{X}_0,\bar{X}_1,\ldots,\bar{X}_{N_l},\bar{x}_1,\ldots,\bar{x}_{N_f})\in L^2_{\cF}(0,T;\hR^n)\times L^2_{\cF}(0,T;\hR^n)\times\ldots\times L^2_{\cF}(0,T;\hR^n).$$
From \eqref{N1}, by using Burkholder-Davis-Gundy (BDG) inequality, there exists a constant $M$, independent of $N_l$ and $N_f$, such that for any $t\in[0,T]$,
\begin{equation*}\begin{aligned}
\hE\Big[\sup_{0\leq s\leq t}\big|\bar{X}_0(s)\big|^2\Big]\leq& M+M\hE\Big[\int_{0}^{t}\big|\bar{X}_0(s)\big|^2+\big|\bar{X}^{(N_l)}(s)\big|^2+\big|\bar{x}^{(N_f)}(s)\big|^2\ud s\Big]\\
\leq&M+M\hE\Big[\int_{0}^{t}\big|\bar{X}_0(s)\big|^2+\frac{1}{N_l}\sum_{i=1}^{N_l}\big|\bar{X}_i(s)\big|^2+\frac{1}{N_f}\sum_{j=1}^{N_f}\big|\bar{x}_j(s)\big|^2\ud s\Big]
\end{aligned}\end{equation*}
and by Gronwall's inequality, we obtain
\begin{equation}\label{sec5:1}
\hE\Big[\sup_{0\leq s\leq t}\big|\bar{X}_0(s)\big|^2\Big]\leq M+M\hE\Big[\int_{0}^{t}\frac{1}{N_l}\sum_{i=1}^{N_l}\big|\bar{X}_i(s)\big|^2+\frac{1}{N_f}\sum_{j=1}^{N_f}\big|\bar{x}_j(s)\big|^2\ud s\Big].
\end{equation}
Similarly, we have
\begin{equation}\label{sec5:2}
\hE\Big[\sup_{0\leq s\leq t}\big|\bar{X}_i(s)\big|^2\Big]\leq M+M\hE\Big[\int_{0}^{t}\big|\bar{X}_i(s)\big|^2+\frac{1}{N_l}\sum_{i=1}^{N_l}\big|\bar{X}_i(s)\big|^2\ud s\Big],\qq 1\leq i\leq N_l,
\end{equation}
and
\begin{equation}\label{sec5:3}
\hE\Big[\sup_{0\leq s\leq t}\big|\bar{x}_j(s)\big|^2\Big]\leq M+M\hE\Big[\int_{0}^{t}\big|\bar{x}_j(s)\big|^2+\frac{1}{N_f}\sum_{j=1}^{N_f}\big|\bar{x}_j(s)\big|^2\ud s\Big],\qq 1\leq j\leq N_f.
\end{equation}
Thus
\begin{equation*}
\hE\Big[\sup_{0\leq s\leq t}\sum_{i=1}^{N_l}\big|\bar{X}_i(s)\big|^2\Big]\leq \hE\Big[\sum_{i=1}^{N_l}\sup_{0\leq s\leq t}\big|\bar{X}_i(s)\big|^2\Big]\leq M N_l+2M\hE\Big[\int_{0}^{t}\sum_{i=1}^{N_l}\big|\bar{X}_i(s)\big|^2\Big],
\end{equation*}
and
\begin{equation*}
\hE\Big[\sup_{0\leq s\leq t}\sum_{j=1}^{N_f}\big|\bar{x}_j(s)\big|^2\Big]\leq \hE\Big[\sum_{j=1}^{N_f}\sup_{0\leq s\leq t}\big|\bar{x}_j(s)\big|^2\Big]\leq M N_f+2M\hE\Big[\int_{0}^{t}\sum_{j=1}^{N_f}\big|\bar{x}_j(s)\big|^2\Big].
\end{equation*}
By Gronwall's inequality, it follows that $\hE\Big[\sup_{0\leq s\leq t}\sum_{i=1}^{N_l}\big|\bar{X}_i(s)\big|^2\Big]=O(N_l)$, and $\hE\Big[\sup_{0\leq s\leq t}\sum_{j=1}^{N_f}\big|\bar{x}_j(s)\big|^2\Big]=O(N_f)$. Then, substituting this estimate to \eqref{sec5:2} and \eqref{sec5:3} and Gronwall's inequality yields $\hE\Big[\sup_{0\leq s\leq t}\big|\bar{X}_i(s)\big|^2\Big]\leq M$, and $\hE\Big[\sup_{0\leq s\leq t}\big|\bar{x}_j(s)\big|^2\Big]\leq M$. By applying this estimate to \eqref{sec5:1}, we get $\hE\Big[\sup_{0\leq s\leq t}\big|\bar{X}_0(s)\big|^2\Big]\leq M$.
\end{proof}

Now, we recall that
\begin{equation*}
\bar{X}^{(N_l)}(t)=\frac{1}{N_l}\sum_{i=1}^{N_l}\bar{X}_i(t),\qq\mathrm{and}\qq\bar{x}^{(N_f)}(t)=\frac{1}{N_f}\sum_{j=1}^{N_f}\bar{x}_j(t),
\end{equation*}
then we have
\begin{lemma}\label{l5.2}
Under assumptions {\rm(H1)-(H2)}, and if the conditions in the Theorem \ref{C6l1}, Theorem \ref{C6l2}, Theorem \ref{C6l3} hold, then there exists a constant $M$ independent of $N_l$ and $N_f$, such that
$$\hE\Big[\sup_{0\leq t\leq T}\big|\bar{X}^{(N_l)}(t)-\bar{m}_X(t)\big|^2\Big]\leq \frac{M}{N_l},$$
$$\hE\Big[\sup_{0\leq t\leq T}\big|\bar{x}^{(N_f)}(t)-\bar{m}_x(t)\big|^2\Big]\leq \frac{M}{N_f}.$$
\end{lemma}
\begin{proof}
For the first one, we have
\begin{equation}\left\{\begin{aligned}\label{N2}
&\ud\Big(\bar{X}^{(N_l)}-\bar{m}_X\Big)=(A+E_1)\Big(\bar{X}^{(N_l)}-\bar{m}_X\Big)\ud t+\frac{1}{N_l}\sum_{i=1}^{N_l}[C\bar{X}_i-DR^{-1}(B^\top \bar{Y}_i+D^\top\bar{Z}_i)+E_2 \bar{X}^{(N_l)}]\ud W_i\\
&\Big(\bar{X}^{(N_l)}-\bar{m}_X\Big)(0)=0.
\end{aligned}\right.\end{equation}
From \eqref{N2}, by using Burkholder-Davis-Gundy (BDG) inequality and Lemma \ref{l5.1}, there exists a constant $M$, independent of $N_l$ and $N_f$, such that for any $t\in[0,T]$,
\begin{equation*}
\hE\Big[\sup_{0\leq s\leq t}\big|\bar{X}^{(N_l)}-\bar{m}_X\big|^2(s)\Big]\leq \frac{M}{N_l}+M\hE\Big[\int_{0}^{t}\big|\bar{X}^{(N_l)}-\bar{m}_X\big|^2(s)\ud s\Big],
\end{equation*}
and by Gronwall's inequality, we obtain
\begin{equation*}
\hE\Big[\sup_{0\leq s\leq t}\big|\bar{X}^{(N_l)}-\bar{m}_X\big|^2(s)\Big]\leq \frac{M}{N_l}.
\end{equation*}
By the same way, we can prove the second formula.
\end{proof}

\begin{lemma}\label{l5.3}
Under assumptions {\rm(H1)-(H2)}, and if the conditions in the Theorem \ref{C6l1}, Theorem \ref{C6l2}, Theorem \ref{C6l3} hold, then there exists a constant $M$ independent of $N_l$ and $N_f$, we have
\begin{equation*}\begin{aligned}
&\Big|\cJ_0(\bar{u}_0,\bar{u},\bar{v})-J_0(\bar{u}_0)\Big|=O\Big(\frac{1}{\sqrt{N}}\Big),\\
&\Big|\cJ_i^l(\bar{u}_0,\bar{u}_i,\bar{u}_{-i})-J_i^l(\bar{u}_0,\bar{u}_i)\Big|=O\Big(\frac{1}{\sqrt{N}}\Big),\\
&\Big|\cJ_j^f(\bar{u}_0,\bar{u},\bar{v}_j,\bar{v}_{-j})-J_j^f(\bar{u}_0,\bar{v}_j)\Big|=O\Big(\frac{1}{\sqrt{N}}\Big),\\
\end{aligned}\end{equation*}
where $N:=\min\{N_l,N_f\}$.
\end{lemma}
\begin{proof}
Let us first consider the major leader agent. Recall \eqref{C6e4} and \eqref{C6e7}, we have
\begin{equation}\label{N3}\begin{aligned}
&\cJ_0(\bar{u}_0,\bar{u},\bar{v})-J_0(\bar{u}_0)\\
=&\frac{1}{2}\hE\Big\{\int_0^T \Big[\Big\|\bar{X}_0-\big(\lambda_0\bar{X}^{(N_l)}+(1-\lambda_0)\bar{x}^{(N_f)}\big)\Big\|_{Q_0}^2 -\Big\|\bar{X}_0-\big(\lambda_0\bar{m}_X+(1-\lambda_0)\bar{m}_x\big)\Big\|_{Q_0}^2\Big]\ud t\Big\},\\
=&\hE\Big\{\int_0^T \Big\langle Q_0\Big(\bar{X}_0-\big(\lambda_0\bar{m}_X+(1-\lambda_0)\bar{m}_x\big)\Big),
\lambda_0(\bar{m}_X-\bar{X}^{(N_l)})+(1-\lambda_0)(\bar{m}_x-\bar{x}^{(N_f)})\Big\rangle \ud t\Big\}\\
&+\frac{1}{2}\hE\Big\{\int_0^T \Big\|\lambda_0(\bar{X}^{(N_l)}-\bar{m}_X)+(1-\lambda_0)(\bar{x}^{(N_f)}-\bar{m}_x)\Big\|_{Q_0}^2\ud t\Big\}.\\
\end{aligned}\end{equation}
By H\"older inequality and Lemma \ref{l5.1}, there exists a constant $M$ independent of $N_l$ and $N_f$ such that
\begin{equation}\label{N4}\begin{aligned}
&\hE\Big\{\int_0^T \Big\langle Q_0\Big|\bar{X}_0-\big(\lambda_0\bar{m}_X+(1-\lambda_0)\bar{m}_x\big)\Big|,
\Big|\lambda_0(\bar{m}_X-\bar{X}^{(N_l)})+(1-\lambda_0)(\bar{m}_x-\bar{x}^{(N_f)})\Big|\Big\rangle \ud t\Big\}\\
\leq&\hE\Big\{\int_0^T \Big|\bar{X}_0-\big(\lambda_0\bar{m}_X+(1-\lambda_0)\bar{m}_x\big)\Big|^2\ud t\Big\}^{\frac{1}{2}}\\
&\hE\Big\{\int_0^T \Big|Q_0\Big(\lambda_0(\bar{m}_X-\bar{X}^{(N_l)})+(1-\lambda_0)(\bar{m}_x-\bar{x}^{(N_f)})\Big)\Big|^2\ud t\Big\}^{\frac{1}{2}}\\
\leq&M\hE\Big\{\int_0^T \Big|Q_0\Big(\lambda_0(\bar{m}_X-\bar{X}^{(N_l)})+(1-\lambda_0)(\bar{m}_x-\bar{x}^{(N_f)})\Big)\Big|^2\ud t\Big\}^{\frac{1}{2}}.\\
\end{aligned}\end{equation}
Noting \eqref{N3}, \eqref{N4} and Lemma \ref{l5.2}, there exists a constant $M$ independent of $N_l$ and $N_f$ such that
\begin{equation}\label{N5}\begin{aligned}
&\hE\Big\{\int_0^T \Big|Q_0\Big(\lambda_0(\bar{m}_X-\bar{X}^{(N_l)})+(1-\lambda_0)(\bar{m}_x-\bar{x}^{(N_f)})\Big)\Big|^2\ud t\Big\}^{\frac{1}{2}}\\
\leq&\Big\{\hE\Big[\sup_{0\leq s\leq t}\big|\bar{X}^{(N_l)}-\bar{m}_X\big|^2(s)\Big]\int_{0}^{T}|Q_0\lambda_0|^2\ud t\Big\}^{\frac{1}{2}} \Big\{\hE\Big[\sup_{0\leq s\leq t}\big|\bar{x}^{(N_f)}-\bar{m}_x\big|^2(s)\Big]\int_{0}^{T}|1-\lambda_0|^2\ud t\Big\}^{\frac{1}{2}}\\
\leq&M\Big(\frac{1}{\sqrt{N_l}}+\frac{1}{\sqrt{N_f}}\Big)=O\Big(\frac{1}{\sqrt{N}}\Big).
\end{aligned}\end{equation}
The rest two claims can be proved in the same way.
\end{proof}

\begin{remark}
We denote $M$ the common constant of the different boundaries. In the above lemmas, the constant $M$ may vary each line by line but it is always independent of the number of minor-leader agents $N_l$ and the number of follower agents $N_f$.
\end{remark}

\subsection{Major leader agent's perturbation}
In this subsection, we will prove that the control strategies set $(\bar{u}_0,\bar{u}_1,\ldots,\bar{u}_{N_l},\bar{v}_1,\ldots,\bar{v}_{N_f})$ given by Theorem \ref{thm5.1} is an $\e$-Nash equilibrium of mixed S-MM-MFG for major leader agent, i.e. there exists an $\e=\e(N)\geq0$, $\lim_{N\rightarrow\infty}\e(N)=0$ such that

$$\cJ_0(\bar{u}_0,\bar{u},\bar{v})\leq\cJ_0(u_0,\bar{u},\bar{v})+\e,\qq\forall u_0\in\cU_0^c[0,T].$$

Let us consider that the major leader agent $\cA_0$ uses an alternative strategy $u_0$, each minor leader agent $\cA_i^l$ uses the control $\bar{u}_i=-R^{-1}(t)(B(t)^\top \bar{Y}_i(t)+D(t)^\top \bar{Z}_i(t))$ and each follower agent $\cA_j^f$ uses the control $\bar{v}_j=-\tilde{R}^{-1}(t)(\tilde{B}(t)^\top \bar{y}_j(t)+\tilde{D}(t)^\top \bar{z}_j(t))$. To prove $(\bar{u}_0,\bar{u}_1,\ldots,\bar{u}_{N_l},\bar{v}_1,\ldots,\bar{v}_{N_f})$ is an $\e$-Nash equilibrium for the major leader agent, we need to show that for possible alternative control $u_0$, $\inf_{u_0\in\cU_0^c[0,T]}\cJ_0(u_0,\bar{u},v[u_0])\geq\cJ_0(\bar{u}_0,\bar{u},\bar{v})-\e$. Then we only need to consider the perturbation $u_0\in\cU_0^c[0,T]$ such that $\cJ_0(u_0,\bar{u},v[u_0])\leq\cJ_0(\bar{u}_0,\bar{u},\bar{v})$. By the representation of the cost functional in \cite{YZ99,SLY16}, we can give the representation of the cost functional as follows.
\begin{proposition}\label{prop5.1}
Let {\rm(H1)-(H2)} hold. There exists a bounded self-adjoint linear operator $N_0:\cU_0^c[0,T]\rightarrow\cU_0^c[0,T]$, a bounded linear operator $H_0:\hR^n\rightarrow\cU_0^c[0,T]$, a bounded real-valued function $M_0:\hR^n\rightarrow\hR$ such that
\begin{equation*}
\cJ_0(x_0,x,y;u_0,u[u_0],v[u_0])=\frac{1}{2}\Big\{\Big\langle N_0u_0(\cd),u_0(\cd)\Big\rangle+2\Big\langle H_0(x_0),u_0(\cd)\Big\rangle+M_0(x_0)\Big\},\ \forall(x_0,u_0)\in\hR^n\times\cU_0^c[0,T].
\end{equation*}
\end{proposition}
\begin{proof}
Refer to Proposition 3.1 in \cite{SLY16}.
\end{proof}
So if we assume that $N_0\gg0$, from Lemma \ref{l5.3}, then there exists a bounded constant $c$, such that
\begin{equation*}
\hE\Big[\int_{0}^{T}\Big|N_0^{\frac{1}{2}}u_0(t)+N_0^{-\frac{1}{2}}H_0(x_0)\Big|^2\ud t\Big]
\leq\cJ_0(u_0,\bar{u},\bar{v})+c\leq\cJ_0(\bar{u}_0,\bar{u},\bar{v})+c\leq J_0(\bar{u}_0)+c+O\Big(\frac{1}{\sqrt{N}}\Big),
\end{equation*}
which implies that $\hE\Big[\int_{0}^{T}\big|u_0(t)\big|^2\ud t\Big]\leq M$, where $M$ is a constant independent of $N$. In fact, by bounded inverse theorem, $N_0^{-1}$ is bounded, so there exists a constant $0<\gamma\leq\|N_0^{\frac{1}{2}}\|$, such that
$$\gamma\hE\Big[\int_{0}^{T}\big|u_0(t)\big|^2\ud t\Big]\leq\|N_0^{\frac{1}{2}}\|\hE\Big[\int_{0}^{T}\Big|u_0(t)+N_0^{-1}H_0(x_0)\Big|^2\ud t\Big]\leq J_0(\bar{u}_0)+c+O\Big(\frac{1}{\sqrt{N}}\Big).$$ Then we have $\hE\Big[\int_{0}^{T}\big|u_0(t)\big|^2\ud t\Big]\leq M$. Similar to Lemma \ref{l5.1}, we can show that
\begin{equation}\label{N6}
\hE\Big[\sup_{0\leq t\leq T}\big|X_0(t)\big|^2\Big]\leq M.
\end{equation}

\begin{lemma}\label{l5.4}
Under assumptions {\rm(H1)-(H2)}, and if the conditions in the Theorem \ref{C6l1}, Theorem \ref{C6l2}, Theorem \ref{C6l3} hold, for the major leader agent's perturbation control $u_0$, we have
$$\Big|\cJ_0(u_0,\bar{u},\bar{v})-J_0(u_0)\Big|=O\Big(\frac{1}{\sqrt{N}}\Big).$$
\end{lemma}
\begin{proof}
Recall \eqref{C6e4} and \eqref{C6e7}, we have
\begin{equation}\label{N7}\begin{aligned}
&\cJ_0(u_0,\bar{u},\bar{v})-J_0(u_0)\\
=&\frac{1}{2}\hE\Big\{\int_0^T \Big[\Big\|X_0-\big(\lambda_0\bar{X}^{(N_l)}+(1-\lambda_0)\bar{x}^{(N_f)}\big)\Big\|_{Q_0}^2 -\Big\|X_0-\big(\lambda_0\bar{m}_X+(1-\lambda_0)\bar{m}_x\big)\Big\|_{Q_0}^2\Big]\ud t\Big\},\\
=&\hE\Big\{\int_0^T \Big\langle Q_0\Big(X_0-\big(\lambda_0\bar{m}_X+(1-\lambda_0)\bar{m}_x\big)\Big),
\lambda_0(\bar{m}_X-\bar{X}^{(N_l)})+(1-\lambda_0)(\bar{m}_x-\bar{x}^{(N_f)})\Big\rangle \ud t\Big\}\\
&+\frac{1}{2}\hE\Big\{\int_0^T \Big\|\lambda_0(\bar{X}^{(N_l)}-\bar{m}_X)+(1-\lambda_0)(\bar{x}^{(N_f)}-\bar{m}_x)\Big\|_{Q_0}^2\ud t\Big\}.\\
\end{aligned}\end{equation}
By H\"older inequality and \eqref{N6}, there exists a constant $M$ independent of $N_l$ and $N_f$ such that
\begin{equation}\label{N8}{\small\begin{aligned}
&\hE\Big\{\int_0^T\Big\langle \Big|Q_0\Big(X_0-\big(\lambda_0\bar{m}_X+(1-\lambda_0)\bar{m}_x\big)\Big)\Big|,
\Big|\lambda_0(\bar{m}_X-\bar{X}^{(N_l)})+(1-\lambda_0)(\bar{m}_x-\bar{x}^{(N_f)})\Big|\Big\rangle\ud t\Big\}\\
\leq&\hE\Big\{\int_0^T \Big|X_0-\big(\lambda_0\bar{m}_X+(1-\lambda_0)\bar{m}_x\big)\Big|^2\ud t\Big\}^{\frac{1}{2}}
\hE\Big\{\int_0^T \Big|Q_0\Big(\lambda_0(\bar{m}_X-\bar{X}^{(N_l)})+(1-\lambda_0)(\bar{m}_x-\bar{x}^{(N_f)})\Big)\Big|^2\ud t\Big\}^{\frac{1}{2}}\\
\leq&M\hE\Big\{\int_0^T \Big|Q_0\Big(\lambda_0(\bar{m}_X-\bar{X}^{(N_l)})+(1-\lambda_0)(\bar{m}_x-\bar{x}^{(N_f)})\Big)\Big|^2\ud t\Big\}^{\frac{1}{2}}.\\
\end{aligned}}\end{equation}
At last, same as the Lemma \ref{l5.3}, noting \eqref{N7}, \eqref{N8}, and Lemma \ref{l5.2}, there exists a constant $M$ independent of $N_l$ and $N_f$ such that
\begin{equation}\label{N9}\begin{aligned}
&\hE\Big\{\int_0^T \Big|Q_0\Big(\lambda_0(\bar{m}_X-\bar{X}^{(N_l)})+(1-\lambda_0)(\bar{m}_x-\bar{x}^{(N_f)})\Big)\Big|^2\ud t\Big\}^{\frac{1}{2}}\\
\leq&\Big\{\hE\Big[\sup_{0\leq s\leq t}\big|\bar{X}^{(N_l)}-\bar{m}_X\big|^2(s)\Big]\int_{0}^{T}|Q_0\lambda_0|^2\ud t\Big\}^{\frac{1}{2}} \Big\{\hE\Big[\sup_{0\leq s\leq t}\big|\bar{x}^{(N_f)}-\bar{m}_x\big|^2(s)\Big]\int_{0}^{T}|1-\lambda_0|^2\ud t\Big\}^{\frac{1}{2}}\\
\leq&M\Big(\frac{1}{\sqrt{N_l}}+\frac{1}{\sqrt{N_f}}\Big)=O\Big(\frac{1}{\sqrt{N}}\Big).
\end{aligned}\end{equation}
\end{proof}

Taking the advantage of Lemma \ref{l5.3} and Lemma \ref{l5.4}, we can give the first part of the proof to the Theorem \ref{thm5.1}, i.e. the control strategies set $(\bar{u}_0,\bar{u}_1,\ldots,\bar{u}_{N_l},\bar{v}_1,\ldots,\bar{v}_{N_f})$ given by Theorem \ref{thm5.1} is an $\e$-Nash equilibrium of the mixed S-MM-MFG for major leader agent.

\medskip

\textbf{Part A of the proof to Theorem \ref{thm5.1}}

\medskip
\noindent Combining Lemma \ref{l5.3} and Lemma \ref{l5.4}, we have
$$\cJ_0(\bar{u}_0,\bar{u},\bar{v})\leq J_0(\bar{u}_0)+O\Big(\frac{1}{\sqrt{N}}\Big)\leq J_0(u_0)+O\Big(\frac{1}{\sqrt{N}}\Big)\leq \cJ_0(u_0,\bar{u},\bar{v})+O\Big(\frac{1}{\sqrt{N}}\Big),$$
where the second inequality comes from the fact that $J_0(\bar{u}_0)=\inf_{u_0\in\cU_0^c[0,T]}J_0(u_0)$. Consequently, the Theorem \ref{thm5.1} holds for the major leader agent with $\e=O\Big(\frac{1}{\sqrt{N}}\Big)$.
\endproof

\subsection{Minor leader agent's perturbation}
Now, let us consider the following case: a given minor leader agent $\cA_i^l$ uses an alternative strategy $u_i\in\cU_i^c[0,T]$, the major leader agent uses $\bar{u}_0$, each follower agent $\cA_j^f$ uses $\bar{v}_j$ while other minor leader agents use the control $\bar{u}_{-i}$. In fact, by the representation of the cost functional (which is similar to the argument of major leader agent), to prove $(\bar{u}_0,\bar{u}_1,\ldots,\bar{u}_{N_l},\bar{v}_1,\ldots,\bar{v}_{N_f})$ is an $\e$-Nash equilibrium for the minor leader agent, we only need to consider the perturbation $u_i\in\cU_i^c[0,T]$ satisfying
$$\hE\Big[\int_{0}^{T}\big|u_i(t)\big|^2\ud t\Big]\leq M,$$
where $M$ is a constant independent of $N_l$. Then similar to Lemma \ref{l5.1}, we can show that
\begin{equation}\label{N10}
\sup_{1\leq i\leq N_l}\hE\Big[\sup_{0\leq t\leq T}\big|X_i(t)\big|^2\Big]\leq M.
\end{equation}

\begin{lemma}\label{l5.5}
Under assumptions {\rm(H1)-(H2)}, and if the conditions in the Theorem \ref{C6l1}, Theorem \ref{C6l2}, Theorem \ref{C6l3} hold, then there exists a constant $M$ independent of $N_l$ and $N_f$, such that
$$\hE\Big[\sup_{0\leq t\leq T}\big|X^{(i,N_l)}(t)-\bar{m}_X(t)\big|^2\Big]\leq \frac{M}{N_l},$$
where $X^{(i,N_l)}(t)=\frac{1}{N_l}\big(X_i(t)+\sum_{k\neq i}\bar{X}_k(t)\big)$.
\end{lemma}
\begin{proof}
In fact, we have
$$X^{(i,N_l)}(t)-\bar{X}^{(N_l)}(t)=\frac{1}{N_l}X_i(t),$$
by \eqref{N10} , it yields
$$\hE\Big[\sup_{0\leq t\leq T}\big|X^{(i,N_l)}(t)-\bar{X}^{(N_l)}(t)\big|^2\Big]\leq \frac{M}{N_l}.$$
Combined with Lemma \ref{l5.2}, we can directly get
$$\hE\Big[\sup_{0\leq t\leq T}\big|X^{(i,N_l)}(t)-\bar{m}_X(t)\big|^2\Big]\leq \frac{M}{N_l}.$$
\end{proof}

\begin{lemma}\label{l5.6}
Under assumptions {\rm(H1)-(H2)}, and if the conditions in the Theorem \ref{C6l1}, Theorem \ref{C6l2}, Theorem \ref{C6l3} hold, for the minor leader agent's perturbation control $u_i$, we have
$$\Big|\cJ_i^l(\bar{u}_0,u_i,\bar{u}_{-i})-J_i^l(\bar{u}_0,u_i)\Big|=O\Big(\frac{1}{\sqrt{N}}\Big).$$
\end{lemma}
\begin{proof}
Recall \eqref{C6e5} and \eqref{C6e8}, we have
\begin{equation}\label{N11}\begin{aligned}
&\cJ_i^l(\bar{u}_0,u_i,\bar{u}_{-i})-J_i^l(\bar{u}_0,u_i)\\
=&\frac{1}{2}\hE\Big\{\int_0^T \Big[\Big\|X_i-\big(\lambda\bar{X}^{(i,N_l)}+(1-\lambda)X_0\big)\Big\|_Q^2 -\Big\|X_i-\big(\lambda\bar{m}_X+(1-\lambda)X_0\big)\Big\|_Q^2\Big]\ud t\Big\},\\
=&\hE\Big\{\int_0^T \Big\langle Q\Big(X_i-\big(\lambda\bar{m}_X+(1-\lambda)X_0\big)\Big),\lambda(\bar{m}_X-\bar{X}^{(i,N_l)})\Big\rangle\ud t\Big\} +\frac{1}{2}\hE\Big\{\int_0^T \Big\|\lambda(\bar{X}^{(i,N_l)}-\bar{m}_X)\Big\|_Q^2\ud t\Big\}.\\
\end{aligned}\end{equation}
By same technique, applying H\"older inequality, Lemma \ref{l5.5}, and \eqref{N10}, there exists a constant $M$ independent of $N_l$ and $N_f$ such that
\begin{equation}\label{N12}\begin{aligned}
&\hE\Big\{\int_0^T \Big\langle\Big|Q\Big(X_i-\big(\lambda\bar{m}_X+(1-\lambda)X_0\big)\Big)\Big|,\Big|\lambda(\bar{m}_X-\bar{X}^{(i,N_l)})\Big|\Big\rangle\ud t\Big\}\\
\leq&\hE\Big\{\int_0^T \Big|X_i-\big(\lambda\bar{m}_X+(1-\lambda)X_0\big)\Big|^2\ud t\Big\}^{\frac{1}{2}}\hE\Big\{\int_0^T \Big|Q\lambda(\bar{m}_X-\bar{X}^{(i,N_l)})\Big|^2\ud t\Big\}^{\frac{1}{2}}\\
\leq&M\hE\Big\{\int_0^T \Big|Q\lambda(\bar{m}_X-\bar{X}^{(i,N_l)})\Big|^2\ud t\Big\}^{\frac{1}{2}}
\leq M\Big\{\hE\Big[\sup_{0\leq s\leq t}\big|\bar{X}^{(i,N_l)}-\bar{m}_X\big|^2(s)\Big]\int_{0}^{T}|Q\lambda|^2\ud t\Big\}^{\frac{1}{2}}\\
\leq&\frac{M}{\sqrt{N_l}}=O\Big(\frac{1}{\sqrt{N}}\Big).
\end{aligned}\end{equation}
\end{proof}

Taking the advantage of Lemma \ref{l5.3} and Lemma \ref{l5.6}, we can give the second part of the proof to the Theorem \ref{thm5.1}, i.e. the control strategies set $(\bar{u}_0,\bar{u}_1,\ldots,\bar{u}_{N_l},\bar{v}_1,\ldots,\bar{v}_{N_f})$ given by Theorem \ref{thm5.1} is an $\e$-Nash equilibrium of the mixed S-MM-MFG for minor leader agent.

\medskip

\textbf{Part B of the proof to Theorem \ref{thm5.1}}

\medskip
\noindent Combining Lemma \ref{l5.3} and Lemma \ref{l5.6}, we have
$$\cJ_i^l(\bar{u}_0,\bar{u}_i,\bar{u}_{-i})\leq J_i^l(\bar{u}_0,\bar{u}_i)+O\Big(\frac{1}{\sqrt{N}}\Big)\leq J_i^l(\bar{u}_0,u_i)+O\Big(\frac{1}{\sqrt{N}}\Big)\leq \cJ_i^l(\bar{u}_0,u_i,\bar{u}_{-i})+O\Big(\frac{1}{\sqrt{N}}\Big),$$
where the second inequality comes from the fact that $J_i^l(\bar{u}_0,\bar{u}_i)=\inf_{u_i\in\cU_i^c[0,T]}J_i^l(\bar{u}_0,u_i)$. Consequently, the Theorem \ref{thm5.1} holds for the minor leader agent with $\e=O\Big(\frac{1}{\sqrt{N}}\Big)$.
\endproof

\subsection{Follower agent's perturbation}
At last, we consider the following case: a given follower agent $\cA_j^f$ uses an alternative strategy $v_j\in\cV_j^c[0,T]$, the major leader agent uses $\bar{u}_0$, each minor leader agent $\cA_i^l$ uses $\bar{u}_i$ while other follower agents use the control $\bar{v}_{-j}$. In fact, by the representation of the cost functional (which is similar to the argument of major leader agent), to prove $(\bar{u}_0,\bar{u}_1,\ldots,\bar{u}_{N_l},\bar{v}_1,\ldots,\bar{v}_{N_f})$ is an $\e$-Nash equilibrium for the follower agent, we only need to consider the perturbation $v_j\in\cV_j^c[0,T]$ satisfying
$$\hE\Big[\int_{0}^{T}\big|v_j(t)\big|^2\ud t\Big]\leq M,$$
where $M$ is a constant independent of $N$. Then similar to Lemma \ref{l5.1}, we can show that
\begin{equation}\label{N13}
\sup_{1\leq j\leq N_f}\hE\Big[\sup_{0\leq t\leq T}\big|x_j(t)\big|^2\Big]\leq M.
\end{equation}

\begin{lemma}\label{l5.7}
Under assumptions {\rm(H1)-(H2)}, and if the conditions in the Theorem \ref{C6l1}, Theorem \ref{C6l2}, Theorem \ref{C6l3} hold, then there exists a constant $M$ independent of $N_l$ and $N_f$, such that
$$\hE\Big[\sup_{0\leq t\leq T}\big|x^{(j,N_f)}(t)-\bar{m}_x(t)\big|^2\Big]\leq \frac{M}{N_f},$$
where $x^{(j,N_f)}(t)=\frac{1}{N_f}\big(x_j(t)+\sum_{k\neq j}\bar{x}_k(t)\big)$.
\end{lemma}
\begin{proof}
In fact, we have
$$x^{(j,N_f)}(t)-\bar{x}^{(N_f)}(t)=\frac{1}{N_f}x_j(t),$$
by \eqref{N13} , it yields
$$\hE\Big[\sup_{0\leq t\leq T}\big|x^{(j,N_f)}(t)-\bar{x}^{(N_f)}(t)\big|^2\Big]\leq \frac{M}{N_f}.$$
Combined with Lemma \ref{l5.2}, we can directly get
$$\hE\Big[\sup_{0\leq t\leq T}\big|x^{(j,N_f)}(t)-\bar{m}_x(t)\big|^2\Big]\leq \frac{M}{N_f}.$$
\end{proof}

\begin{lemma}\label{l5.8}
Under assumptions {\rm(H1)-(H2)}, and if the conditions in the Theorem \ref{C6l1}, Theorem \ref{C6l2}, Theorem \ref{C6l3} hold, for the follower agent's perturbation control $v_j$, we have
$$\Big|\cJ_j^f(\bar{u}_0,\bar{u},v_j,\bar{v}_{-j})-J_j^f(\bar{u}_0,v_j)\Big|=O\Big(\frac{1}{\sqrt{N}}\Big).$$
\end{lemma}
\begin{proof}
Recall \eqref{C6e6} and \eqref{C6e9}, we have
\begin{equation}\label{N14}\begin{aligned}
&\cJ_j^f(\bar{u}_0,\bar{u},v_j,\bar{v}_{-j})-J_j^f(\bar{u}_0,v_j)\\
=&\frac{1}{2}\hE\Big\{\int_0^T \Big[\Big\|x_j-\big(\tilde{\lambda}_1\bar{X}_0+\tilde{\lambda}_2\bar{X}^{(N_l)}+\tilde{\lambda}_3 x^{(j,N_f)}\big)\Big\|_{\tilde{Q}}^2 -\Big\|x_j-\big(\tilde{\lambda}_1\bar{X}_0+\tilde{\lambda}_2\bar{m}_X+\tilde{\lambda}_3\bar{m}_x\big)\Big\|_{\tilde{Q}}^2\Big]\ud t\Big\},\\
=&\hE\Big\{\int_0^T \Big\langle\tilde{Q}\Big(x_j-\big(\tilde{\lambda}_1\bar{X}_0+\tilde{\lambda}_2\bar{m}_X+\tilde{\lambda}_3\bar{m}_x\big)\Big),
\tilde{\lambda}_2(\bar{m}_X-\bar{X}^{(N_l)})+\tilde{\lambda}_3(\bar{m}_x-x^{(j,N_f)})\Big\rangle\ud t\Big\}\\
&+\frac{1}{2}\hE\Big\{\int_0^T \Big\|\tilde{\lambda}_2(\bar{m}_X-\bar{X}^{(N_l)})+\tilde{\lambda}_3(\bar{m}_x-x^{(j,N_f)})\Big\|_{\tilde{Q}}^2\ud t\Big\}.\\
\end{aligned}\end{equation}
By H\"older inequality and \eqref{N13} there exists a constant $M$ independent of $N_l$ and $N_f$ such that
\begin{equation}\label{N15}\begin{aligned}
&\hE\Big\{\int_0^T \Big\langle\Big|\tilde{Q}\Big(x_j-\big(\tilde{\lambda}_1\bar{X}_0+\tilde{\lambda}_2\bar{m}_X+\tilde{\lambda}_3\bar{m}_x\big)\Big)\Big|,
\Big|\tilde{\lambda}_2(\bar{m}_X-\bar{X}^{(N_l)})+\tilde{\lambda}_3(\bar{m}_x-x^{(j,N_f)})\Big|\Big\rangle\ud t\Big\}\\
\leq&\hE\Big\{\int_0^T \Big|x_j-\big(\tilde{\lambda}_1\bar{X}_0+\tilde{\lambda}_2\bar{m}_X+\tilde{\lambda}_3\bar{m}_x\big)\Big|^2\ud t\Big\}^{\frac{1}{2}}\\
&\hE\Big\{\int_0^T \Big|\tilde{Q}\Big(\tilde{\lambda}_2(\bar{m}_X-\bar{X}^{(N_l)})+\tilde{\lambda}_3(\bar{m}_x-x^{(j,N_f)})\Big)\Big|^2\ud t\Big\}^{\frac{1}{2}}\\
\leq&M\hE\Big\{\int_0^T \Big|\tilde{Q}\Big(\tilde{\lambda}_2(\bar{m}_X-\bar{X}^{(N_l)})+\tilde{\lambda}_3(\bar{m}_x-x^{(j,N_f)})\Big)\Big|^2\ud t\Big\}^{\frac{1}{2}}\\
\end{aligned}\end{equation}
At last, same as the Lemma \ref{l5.3}, noting \eqref{N14}, \eqref{N15}, and Lemma \ref{l5.7}, there exists a constant $M$ independent of $N_l$ and $N_f$ such that
\begin{equation}\label{N16}\begin{aligned}
&\hE\Big\{\int_0^T \Big|\tilde{Q}\Big(\tilde{\lambda}_2(\bar{m}_X-\bar{X}^{(N_l)})+\tilde{\lambda}_3(\bar{m}_x-x^{(j,N_f)})\Big)\Big|^2\ud t\Big\}^{\frac{1}{2}}\\
\leq&\Big\{\hE\Big[\sup_{0\leq s\leq t}\big|\bar{X}^{(N_l)}-\bar{m}_X\big|^2(s)\Big]\int_{0}^{T}|\tilde{Q}\tilde{\lambda}_2|^2\ud t\Big\}^{\frac{1}{2}} \Big\{\hE\Big[\sup_{0\leq s\leq t}\big|x^{(j,N_f)}-\bar{m}_x\big|^2(s)\Big]\int_{0}^{T}|\tilde{\lambda}_3|^2\ud t\Big\}^{\frac{1}{2}}\\
\leq&M\Big(\frac{1}{\sqrt{N_l}}+\frac{1}{\sqrt{N_f}}\Big)=O\Big(\frac{1}{\sqrt{N}}\Big).
\end{aligned}\end{equation}
\end{proof}

Taking the advantage of Lemma \ref{l5.3} and Lemma \ref{l5.8}, we can give the last part of the proof to the Theorem \ref{thm5.1}, i.e. the control strategies set $(\bar{u}_0,\bar{u}_1,\ldots,\bar{u}_{N_l},\bar{v}_1,\ldots,\bar{v}_{N_f})$ given by Theorem \ref{thm5.1} is an $\e$-Nash equilibrium of the mixed S-MM-MFG for follower agent.

\medskip

\textbf{Part C of the proof to Theorem \ref{thm5.1}}

\medskip
\noindent Combining Lemma \ref{l5.3} and Lemma \ref{l5.8}, we have
$$\cJ_j^f(\bar{u}_0,\bar{u},\bar{v}_j,\bar{v}_{-j})\leq J_j^f(\bar{u}_0,\bar{v}_j)+O\Big(\frac{1}{\sqrt{N}}\Big)\leq J_j^f(\bar{u}_0,v_j)+O\Big(\frac{1}{\sqrt{N}}\Big)\leq \cJ_j^f(\bar{u}_0,\bar{u},v_j,\bar{v}_{-j})+O\Big(\frac{1}{\sqrt{N}}\Big),$$
where the second inequality comes from the fact that $J_j^f(\bar{u}_0,\bar{v}_j)=\inf_{v_j\in\cV_j^c[0,T]}J_j^f(\bar{u}_0,v_j)$. Consequently, the Theorem \ref{thm5.1} holds for the follower agent with $\e=O\Big(\frac{1}{\sqrt{N}}\Big)$. Finally, combined with the Part A, Part B, we complete the proof to Theorem \ref{thm5.1}.
\endproof

\section{Special Case}
In this section, we will give an example to show how the major leader influences the whole system. We now look at a special case in which the major leader does not appear. In this case the problem is reduced to a leader-follower mean-field LQG game problem. Let us still regard it as if the major leader does appear but does not affect the game at all, i.e., we assume that
\begin{equation}\label{s1}
  A_0=B_0=C_0=D_0=E_0^1=F_0^1=E_0^2=F_0^2=0,\qq Q_0=0,\qq H_0=0,\qq R_0=I.
\end{equation}
Moreover, let $\l=1$, $\wt{\l}_1=0$, $\wt{\l}_2=\wt{\l}$, and $\wt{\l}_3=1-\wt{\l}$. By observation, we can find that the coupled mean-field term between the leaders and the followers appear on the cost functional of the followers. Thus, it is truly a leader-follower mean-field LQG game problem. By the analysis above, we can get the CC equation of the special case as follows.
\begin{equation}\left\{\begin{aligned}\label{s2}
  &\ud\bar{X}=\{A\bar{X}-BR^{-1}(B^\top \bar{Y}+D^\top \bar{Z})+E_1\hE[\bar{X}]\}\ud t+\{C\bar{X}-DR^{-1}(B^\top \bar{Y}+D^\top \bar{Z})+E_2\hE[\bar{X}]\}\ud W(t)\\
  &\ud\bar{x}=\{\tilde{A}\bar{x}-\tilde{B}\tilde{R}^{-1}(\tilde{B}^\top\bar{y}+\tilde{D}^\top\bar{z})+F_1\hE[\bar{x}]\}\ud t+\{\tilde{C}\bar{x}-\tilde{D}\tilde{R}^{-1}(\tilde{B}^\top\bar{y}+\tilde{D}^\top\bar{z})+F_2\hE[\bar{x}]\}\ud\widetilde{W}(t)\\
  &\ud K=\{\tilde{A}K+\tilde{B}\tilde{R}^{-1}\tilde{B}^\top p+\tilde{B}\tilde{R}^{-1}\tilde{D}^\top q\}\ud t+\{\tilde{C}K+\tilde{D}\tilde{R}^{-1}\tilde{B}^\top p+\tilde{D}\tilde{R}^{-1}\tilde{D}^\top q\}\ud\wt{W}(t)\\
  &\ud\bar{Y}=-\Big\{A^\top\bar{Y}+C^\top\bar{Z}+Q\big(\bar{X}-\hE[\bar{X}]\big)\Big\}\ud t+\bar{Z}\ud W(t),\\
  &\ud\bar{y}=-\Big\{\tilde{A}^\top\bar{y}+\tilde{C}^\top\bar{z}+ \tilde{Q}\Big(\bar{x}-\big(\tilde{\lambda} \hE[\bar{X}]+(1-\tilde{\lambda})\hE[\bar{x}]\big)\Big)\Big\}\ud t+\bar{z}\ud\wt{W}(t),\\
  &\ud p=-\{\tilde{A}^\top p+\tilde{C}^\top q+F_1^{\top}\hE[p]+F_2^{\top}\hE[q]+\tilde{Q}(1-\tilde{\l})\hE[K]-\tilde{Q}K\}\ud t+q\ud\wt{W}(t),\\
  &\bar{X}(0)=\xi,\q \bar{x}(0)=\zeta,\q K(0)=0,\\
  &\bar{Y}(T)=H\bar{X}(T),\q \bar{y}(T)=\tilde{H}\bar{x}(T),\q p(T)=-\tilde{H}K(T).\\
\end{aligned}\right.\end{equation}
Furthermore, we find that the equation of $K(\cdot)$ and $(p(\cdot),q(\cdot))$ is coupled together but decoupled with other equations and if the consistency condition equation admits a unique adapted solution then $K\equiv0$, $p\equiv0$ and $q\equiv0$ is the trivial solution to the equation. Then the CC equation of the special case is simplified as
\begin{equation}\left\{\begin{aligned}\label{s3}
  &\ud\bar{X}=\{A\bar{X}-BR^{-1}(B^\top \bar{Y}+D^\top \bar{Z})+E_1\hE[\bar{X}]\}\ud t+\{C\bar{X}-DR^{-1}(B^\top \bar{Y}+D^\top \bar{Z})+E_2\hE[\bar{X}]\}\ud W(t)\\
  &\ud\bar{x}=\{\tilde{A}\bar{x}-\tilde{B}\tilde{R}^{-1}(\tilde{B}^\top\bar{y}+\tilde{D}^\top\bar{z})+F_1\hE[\bar{x}]\}\ud t+\{\tilde{C}\bar{x}-\tilde{D}\tilde{R}^{-1}(\tilde{B}^\top\bar{y}+\tilde{D}^\top\bar{z})+F_2\hE[\bar{x}]\}\ud\widetilde{W}(t)\\
  &\ud\bar{Y}=-\Big\{A^\top\bar{Y}+C^\top\bar{Z}+Q\big(\bar{X}-\hE[\bar{X}]\big)\Big\}\ud t+\bar{Z}\ud W(t),\\
  &\ud\bar{y}=-\Big\{\tilde{A}^\top\bar{y}+\tilde{C}^\top\bar{z}+ \tilde{Q}\Big(\bar{x}-\big(\tilde{\lambda} \hE[\bar{X}]+(1-\tilde{\lambda})\hE[\bar{x}]\big)\Big)\Big\}\ud t+\bar{z}\ud\wt{W}(t),\\
  &\bar{X}(0)=\xi,\q \bar{x}(0)=\zeta,\\
  &\bar{Y}(T)=H\bar{X}(T),\q \bar{y}(T)=\tilde{H}\bar{x}(T).\\
\end{aligned}\right.\end{equation}
Next, by the decoupling for the open-loop strategy, we get the Riccati equations \eqref{C6e18} and \eqref{C6e19}. However, it is still hard to get the explicit solution to the Riccati equations. So we consider the $1$-dimensional example 5.1 as follows.

\emph{Example 5.1} Let $n=m_1=m_2=m_3=1$, \eqref{s1} holds, and let
\begin{equation*}\left\{\begin{aligned}
&A=\tilde{A}=0,\qq &&B=\tilde{B}=0,\qq &&E_1=F_1=0,\\
&C=\tilde{C}=0, &&D=\tilde{D}=1, &&E_2=F_2=0,\\
&Q=\tilde{Q}=1, &&R=\tilde{R}=1, &&H=\tilde{H}=0,\\
\end{aligned}\right.\end{equation*}
we have the state equation
\begin{equation*}\left\{\begin{aligned}
&\ud X_i(t)=u_i(t)\ud W_i(t),\\
&\ud x_j(t)=v_j(t)\ud \wt{W}_j(t),\\
& X_i(0)=\xi_i,\qq x_j(0)=\zeta_j.
\end{aligned}\right.\end{equation*}
The cost functional reads
\begin{equation*}\begin{aligned}
\cJ_i^l(u_i(\cdot),\textbf{u}_{-i}(\cdot))=&\frac{1}{2}\mathbb{E}\Big\{\int_0^T \Big(\Big|X_i(t)- X^{(N_l)}(t)\Big|^2+|u_i(t)|^2\Big)\ud t\Big\},
\end{aligned}\end{equation*}
and
\begin{equation*}\begin{aligned}
\cJ_j^f(\textbf{u}(\cdot),v_j(\cdot),\textbf{v}_{-j}(\cdot))=&\frac{1}{2}\mathbb{E}\Big\{\int_0^T \Big(\Big|x_j(t)-\big(\tilde{\lambda} X^{(N_l)}(t)+(1-\tilde{\lambda})x^{(N_f)}(t)\big)\Big|^2+|v_j(t)|^2\Big)\ud t\Big\},
\end{aligned}\end{equation*}
and the Riccati equation
\begin{equation}\left\{\begin{aligned}\label{s4}
&\dot{P}+\textbf{Q}=0,\\
&P(T)=0,
\end{aligned}\right.
\qq\qq\mathrm{and}\qq\qq
\left\{\begin{aligned}
&\dot{\Pi}+\textbf{Q}+\bar{\textbf{Q}}=0,\\
&\Pi(T)=0,
\end{aligned}\right.\end{equation}
where
\begin{equation*}
\textbf{Q}=\begin{pmatrix}\begin{smallmatrix}0&0&0&0\\ 0&-1&0&0\\ 0&0&-1&0\\ 0&0&0&1\end{smallmatrix}\end{pmatrix},\qq\qq
\bar{\textbf{Q}}=\begin{pmatrix}\begin{smallmatrix}0&0&0&0\\ 0&1&0&0\\ 0&\tilde{\lambda}&1-\tilde{\lambda}&0\\ 0&0&0&\tilde{\lambda}-1\end{smallmatrix}\end{pmatrix},
\end{equation*}
so it can be easily solved out that
\begin{equation*}
P(t)=\begin{pmatrix}\begin{smallmatrix}0&0&0&0\\ 0&t-T&0&0\\ 0&0&t-T&0\\ 0&0&0&T-t\end{smallmatrix}\end{pmatrix},\qq\qq \Pi(t)=\begin{pmatrix}\begin{smallmatrix}0&0&0&0\\ 0&0&0&0\\ 0&T-\tilde{\l}t&\tilde{\l}t-T&0\\ 0&0&0&T-\tilde{\l}t\end{smallmatrix}\end{pmatrix}.
\end{equation*}
And the optimal control $\bar{u}_i=-\bar{Z}=0$, $\bar{v}_j=-\bar{z}=0$, subject to the optimal cost functional $\cJ_i^l=\frac{1}{2}\xi_i^2$, and $\cJ_j^f=\frac{1}{2}\zeta_j^2$.

\end{document}